\documentclass[11pt]{amsart}

\usepackage{amsfonts,amsmath,amssymb,amsthm,mathrsfs,url,color,latexsym,rawfonts,lscape,graphicx}
\usepackage{hyperref}
\usepackage{enumerate}

 \usepackage[usenames,dvipsnames]{pstricks}
 \usepackage{epsfig}
 \usepackage[normalem]{ulem}
 \usepackage[off]{auto-pst-pdf}
 \usepackage{pst-grad} 
\usepackage{pst-plot,ifplatform,xkeyval,pst-pdf} 
 \usepackage[space]{grffile} 
 \usepackage{etoolbox} 
 \makeatletter 
 \patchcmd\Gread@eps{\@inputcheck#1 }{\@inputcheck"#1"\relax}{}{}
 \makeatother

\theoremstyle{plain}
  \newtheorem{theorem}{Theorem}[section]
  \newtheorem{lemma}{Lemma}[section]

  \newtheorem{corollary}{Corollary}[section]

\theoremstyle{definition}

\theoremstyle{remark} 
  \newtheorem{remark}{Remark}[section]

   \newcommand{\beqn}{\begin{eqnarray}}
   \newcommand{\eeqn}{\end{eqnarray}}
   \newcommand{\beqs}{\begin{eqnarray*}}
   \newcommand{\eeqs}{\end{eqnarray*}}
   \newcommand{\ban}{\begin{eqnarray*}}
   \newcommand{\nan}{\end{eqnarray*}}
   \newcommand{\beq}{\begin{equation}}
   \newcommand{\eeq}{\end{equation}}

  \newcommand{\RR}{{\mathbb R}}

  \newcommand{\R}{\RR}

\newcommand{\p}{\partial}
\newcommand{\eps}{\varepsilon}

\newcommand{\Om}{\Omega}
\newcommand{\pom}{{\p\Om}}
\newcommand{\bom}{{\overline\Om}}

\renewcommand{\det}{\mbox{det}}
  
  \newcommand{\dist}{\mbox{dist}}

  \newcommand{\Vol}{{\mbox{Vol}}}

\newcommand{\lan}{\langle}
\newcommand{\ran}{\rangle}

  \numberwithin{equation}{section}
  \numberwithin{figure}{section}

\begin{document}

\title[Global regularity for the Monge-Amp\`ere equation]
{Global regularity for the Monge-Amp\`ere equation \\ with natural boundary condition}

\date\today

\author[S. Chen]
{Shibing Chen}
\address
{Centre for Mathematics and Its Applications,
The Australian National University,
Canberra, ACT 0200, AUSTRALIA}
\email{chenshibing1982@hotmail.com}

\author[J. Liu]
{Jiakun Liu}
\address
	{
	School of Mathematics and Applied Statistics,
	University of Wollongong,
	Wollongong, NSW 2522, AUSTRALIA}
\email{jiakunl@uow.edu.au}

\author[X.-J. Wang]
{Xu-Jia Wang}
\address
{Centre for Mathematics and Its Applications,
The Australian National University,
Canberra, ACT 0200, AUSTRALIA}
\email{Xu-Jia.Wang@anu.edu.au}

\thanks{This work was supported by ARC FL130100118 and ARC DP170100929.}



\begin{abstract}
In this paper, we establish the global $C^{2,\alpha}$  and $W^{2,p}$ regularity for the Monge-Amp\`ere equation 
$\det\, D^2u = f$ subject to  boundary condition $Du(\Omega) = \Omega^*$, 
where $\Om$ and $\Om^*$ are bounded convex domains in the Euclidean space $\R^n$ with $C^{1,1}$ boundaries,
and $f$ is a H\"older continuous function.
This boundary value problem arises naturally in optimal transportation and many other applications.
\end{abstract}

\maketitle

\baselineskip=15.6pt
\parskip=3pt

\section{Introduction}

In this paper we establish the global $C^{2,\alpha}$ and $W^{2,p}$ regularity for the Monge-Amp\`ere equation
	\begin{equation}\label{MA1}
		\det\, D^2u(x) = f(x)
	\end{equation}
subject to the boundary condition
	\begin{equation}\label{bdry}
		Du(\Omega) = \Omega^* , 
	\end{equation}
where $\Omega$, $\Omega^*$ are bounded convex domains in $\mathbb{R}^n$ with $C^{1,1}$ boundary, and
 $f $ is a positive function. 
 We also assume that $f\in C^0 (\bom)$ for the global $W^{2,p}$ estimate ($p\ge 1$);
 and $f\in C^\alpha(\bom)$ for the global $C^{2,\alpha}$ estimate ($\alpha\in (0,1)$).

The boundary value problem \eqref{MA1} and \eqref{bdry} arises naturally in optimal transportation with the quadratic cost function.
It is a fundamental problem  in the area and received much attention due to its wide range of applications, 
such as in fluid mechanics, meteorology, image recognition, reflector design, and also in geometry and probability \cite{DF, E99, RR, V1, V2}.
In particular it was recently found that 
the problem \eqref{MA1} and \eqref{bdry} plays a  fundamental role in Wasserstein generative adversarial networks, 
a fast growing technique in machine learning \cite{LGA}. 
The existence and uniqueness of solutions to the problem  \eqref{MA1} and \eqref{bdry}
were obtained  by Brenier in his pioneering work \cite{Bre91}.
Since then the regularity of solutions has been a focus of attention in this area \cite {V1, V2}, 
and has been studied in \cite{C92, C92a, C96, D91, U1}.
When $\Om$ and $\Om^*$ are bounded convex domains, and $f\ge 0$ satisfies the doubling condition,
the global $C^{1,\alpha}$ regularity for the solution was obtained by Caffarelli \cite {C92}.
In a landmark paper \cite{C96},  
Caffarelli established the global $C^{2,\alpha'}$ regularity 
for the problem \eqref{MA1} and \eqref{bdry},
assuming that $\Om$ and $\Om^*$ are uniformly convex with $C^2$ boundary, $f\in C^\alpha(\bom)$ and $f>0$.
When $\Om$ and $\Om^*$ are uniformly convex and $C^{3,1}$ smooth, and $f\in C^{1,1}(\bom)$, 
the global smooth solution was first obtained by Delano\"e \cite{D91} in 1991 for dimension two 
and later extended to high dimensions by Urbas \cite{U1}.
The results of Caffarelli, Delano\"e, and Urbas were used by Brendle and Warren \cite {Br1, Br2} to study the 
minimal Lagrangian graphs. 
These results may also be applied to the problem of convex hypersurfaces with prescribed spherical map \cite{P69}.

The uniform convexity of domains is a natural condition for the regularity of solutions 
to boundary value problems of the Monge-Amp\`ere equation. 
In fact, the uniform convexity is a necessary condition for the global regularity
of solutions to the Dirichlet problem \cite {CNS, Sa, TW08a}. 
It was used extensively and played a critical role in the proof 
for both the Dirichlet problem and the second boundary value problem \eqref{MA1}, \eqref{bdry} 
in the above mentioned papers \cite{C96, CNS, D91, Sa, TW08a,U1},
and also in the paper on the Neumann problem \cite{LTU}.

Surprisingly, we found that for the boundary value problem  \eqref{MA1} and \eqref{bdry},
the uniform convexity of domains can be dropped. 
In this paper we obtain the global $C^{2,\alpha}$ regularity for the problem \eqref{MA1} and \eqref{bdry},
assuming both $\Om$ and $\Om^*$ are convex only (instead of uniformly convex).
From \cite{C92a}, \cite[\S7.3]{MTW} it is known that for arbitrary positive and smooth functions $f$, 
the convexity of domains is necessary for the global $C^1$ regularity. 

Not only the uniform convexity of domains can be dropped,  
we also prove that the boundary smoothness can be reduced to $C^{1,1}$. 
Note that if the boundaries are $C^2$ and uniformly convex, 
they will become quadratic polynomials after blowing-up, 
but if the boundaries are only $C^{1,1}$, 
we have to deal with the possibility that limit shape is not even $C^1$ smooth after the blowing-up.
This is the situation that gives rise to substantial difficulties in our proof (see Remark \ref{r42}). 
By blowing-up, we mean to normalise a sequence of sub-level sets of the solution.

Under the above assumptions on domains, in this paper 
we obtain the sharp boundary $C^{2,\alpha}$ regularity 
when $f\in C^\alpha(\bom)$ and $f>0$;
and $C^2$ regularity when $f$ is Dini continuous.
For the Dirichlet problem, the sharp boundary $C^{2,\alpha}$ estimate  was obtained in \cite{TW08a, Sa}, 
and the interior $C^{2,\alpha}$ estimate was obtained in \cite{C1}.

\begin{theorem}\label{main}
Assume that $\Omega$ and $\Omega^*$ are bounded convex domains in $\mathbb{R}^n$ with $C^{1,1}$ boundary, and 
assume that $f\in C^\alpha(\bom)$ is positive, for some $\alpha\in(0,1)$.
Let $u$ be a convex solution to  \eqref{MA1} and \eqref{bdry}. Then we have the estimate
\beq\label{esti}
\|u\|_{C^{2,\alpha}(\bom)}\le C,
\eeq
where $C$ is a constant depending only $n, \alpha, f, \Om$, and $\Om^*$.
\end{theorem}

Our argument also leads to the global $W^{2,p}$ estimate for the solution.

\begin{theorem}\label{main01}
Assume that $\Omega$ and $\Omega^*$ are bounded convex domains in $\mathbb{R}^n$ with $C^{1,1}$ boundary, and 
assume that $f\in C^0(\bom)$ is positive.
Let $u$ be a convex solution to  \eqref{MA1} and \eqref{bdry}. Then we have the estimate
\beq\label{esti01}
\|u\|_{W^{2,p}(\bom)}\le C
\eeq
for all $p\ge 1$,
where $C$ is a constant depending only $n, p, f, \Om$, and $\Om^*$.
\end{theorem}

The interior $W^{2,p}$ estimate for the Monge-Amp\`ere equation was proved by Caffarelli \cite{C1}.
The $W^{2,p}$ estimate at the boundary was obtained by Savin \cite {Sa1} for the Dirichlet problem, 
and by Figalli and the first author \cite{CF} for the boundary condition \eqref{bdry}. 
The proof in \cite{CF} relies on the estimates in \cite {C96} and thus required the domains to be $C^2$ smooth and uniformly convex.
In this paper, we assume that the domains are convex with $C^{1,1}$ boundary, see Remark \ref{r52} as well. 

Relaxing the uniform convexity of domains to convexity does make sense in applications. 
For example, in Wasserstein generative adversarial networks, a typical case is when the domains $\Om$ and $\Om^*$ are squares
or cubes \cite {LGA, SF}. 
Theorems \ref{main} and \ref{main01} imply the regularity of the solution on the faces of the cube.
We will prove the $C^{3,\alpha}$ regularity at the corner in a separate paper. In dimension two, the $C^{3,\alpha}$ regularity was proved in \cite{Jh} and it is optimal.

The proof of Theorems \ref{main} and \ref{main01} is based on delicate analysis on sub-level sets of the solution 
near the boundary and uses various techniques on the Monge-Amp\`ere equation \cite{Fig, Gut}, 
in particular those from Caffarelli's papers.
The uniform density in \S\ref{s2} was introduced by Caffarelli \cite{C96} but a different proof is needed for non-uniformly convex domains.
The key estimate of the paper is the following uniform obliqueness:

\begin{lemma}\label{lem obli}
Assume that $\Omega$, $\Omega^*$ are two convex domains with $C^{1,1}$ boundaries, 
and that $f$ is positive and continuous.
Let $0\in\partial\Omega$ and the image $Du(0)=0\in\partial\Omega^*$. 
Then there exists a positive constant $\mu$ such that 
	\begin{equation}\label{sobli}
		\langle \nu(0), \nu^*(Du(0)) \rangle \geq\mu>0,
	\end{equation}
where $\nu$ and $\nu^*$ are the  unit inner normals of $\Om$ and $\Om^*$, respectively.
\end{lemma}

Lemma \ref{lem obli} will be proved in \S\ref{s4} and \S\ref{ss42}.
To prove \eqref{sobli} for non-uniformly convex domains, we need to introduce a completely different and new idea.
We also provide different proof for the boundary $C^{2,\alpha}$ estimate in \S\ref{s555}, 
for convex domains with $C^{1,1}$ boundary. 
{These new techniques may apply to other problems related to Monge-Amp\`ere type equations. 
In particular, we have recently established the $C^{2,\alpha}$ regularity of free boundaries in optimal transportation \cite{CLW-free}, 
thus resolved an open problem raised by Caffarelli and McCann in \cite{CM}. }

This paper is organised as follows. 
In \S\ref{s2}, we introduce some properties on the sub-level sets of solutions to the problem \eqref{MA1} and \eqref{bdry},
and prove the uniform density property. 
In \S\ref{s3}, we obtain the tangential $C^{1,\alpha}$ regularity for any given $\alpha\in(0,1)$.
In \S\ref{s4} and \S\ref{ss42}, we prove the uniform obliqueness in dimension two and high dimensions, respectively, 
which is the key ingredient for the proof of the global $C^{2,\alpha}$  and $W^{2,p}$ regularity.
Finally in \S\ref{s555} we complete the proof of Theorems \ref{main} and \ref{main01}.

\section{Uniform density}\label{s2}
 
Consider the optimal transport with density $f$ in $\Omega$ and  density $1$ in $\Omega^*$.
We assume that $f$ satisfies $\lambda^{-1}<f<\lambda$ for a constant $\lambda>0$ 
and $\int_\Omega f(x) dx = \int_{\Omega^*} dy$.
Let $u$ and $v$ be the potential functions in $\Omega$ and $\Omega^*$, respectively. 
Then $u$ is a solution to \eqref{MA1} and \eqref{bdry}.
We extend $u, v$ to convex functions in $\mathbb{R}^n$ as follows:
	\begin{align*}
\tilde{u}(x):=\sup\{\ell(x) :\ \ell \text{ is affine, } \ell\leq u\ \text{in}\ \Omega,\ \nabla\ell\in \Omega^*\} \quad \mbox{ for }x\in\mathbb{R}^n; \\
\tilde{v}(y):=\sup\{\ell(y) :\  \ell \text{ is affine, }\ \ell\leq v\ \text{in}\ \Omega^*,\ \nabla\ell\in \Omega\} \quad \mbox{ for }y\in\mathbb{R}^n.
	\end{align*}
For simplicity of notations, we denote the extended functions $\tilde{u}, \tilde{v}$ as $u, v$.
Let $0\in\partial\Omega$ be a boundary point. 
By subtracting a linear function, we assume that $u(0) = 0$ and $u\geq 0$. 
Correspondingly, one has $0\in\partial\Omega^*$, $v(0) = 0$ and $v\geq 0$ as well. 

We introduce two different sub-level sets of $u$ at $x_0\in\overline\Om$.
One is 
$$S_h[u](x_0) = \Big\{ x\in\Omega :  u(x) <\ell_{x_0} (x)+ h \Big\},$$ 
which may be abbreviated as $S_h[u]$ or $S_h(x_0)$ when no confusion arises,
where $\ell_{x_0}$ is a support function of $u$ at $x_0$.
The other one is the \emph{centred} sub-level set
$$S^c_h[u](x_0)=\Big\{ x\in\mathbb{R}^n : u(x)< \hat\ell (x) + h \Big\} $$ 
or simply denoted as $S^c_h[u]$ or $S^c_h(x_0)$, where the affine function $\hat \ell $ is chosen such that $\hat\ell(x_0)=u(x_0)$
and $x_0$ is the mass centre for $S^c_h[u](x_0)$.
The existence of such a linear function is proved in \cite{C92}. 
Note that $S_h(x_0)$ is contained in $\Om$, but $S^c_h(x_0)$ may contain both points in and out of $\Om$.

The extended function $u\in C^1(\R^n)$ and satisfies $\det\,D^2u=f\chi_\Om$ in $\R^n$ if $\Om, \Om^*$ are convex. 
The following lemma was established by Caffarelli \cite[Corollary 2.2]{C96}.

\begin{lemma}\label{Lemma2.2}
Assume that $\Om, \Om^*$ are convex and bounded. 
Given a centred sub-level set $S^c_h(x_0)$ with $x_0\in\overline\Om$,
let $T$ be a linear transform such that $B_1(0) \subset S^*=:T(S^c_h(x_0)) \subset B_n (0)$.
Then $\hat u(x)=: h^{-1} [u-\ell](T^{-1}(x))$ satisfies 
\beq\label{good2}
B_r(0) \subset \nabla \hat u(\frac 12 S^*)\subset \hat u( S^*)\subset B_{r^{-1}} (0) , 
\eeq
where $\ell$ is the linear function such that $u=\ell$ on  $\p S^c_h(x_0)$.
Scaling back, there is an ellipsoid $E$ centred at $\nabla \ell$
such that
\beq\label{good1}
 r E\subset \nabla u(S^c_h(x_0))\subset  r^{-1} E,
\eeq
where 
$\alpha E$ denotes the $\alpha$-dilation of $E$ with respect to its centre, 
and the constant $r>0$ depends only on $n, \lambda$, $\Om, \Om^*$, but is independent of $h$ and $u$.
\end{lemma}

Lemma \ref{Lemma2.2} implies that $\nabla \ell$  is a true interior point of $\nabla u(S^c_h(x_0))$, namely
it has a positive distance from the boundary after normalisation.
The first inclusion in \eqref{good2} also follows from the strict convexity of $u$ \cite[Corollary 2.3]{C96}, namely
\beq\label{stricon}
u(x)\ge {u(x_0)+} \nabla u(x_0) (x-x_0)+ c_0 h \ \ \ \forall\ x\in \p S^c_h(x_0)\cap \Om,
\eeq
where $c_0$ is a constant depending on $n, \lambda$, $\Om, \Om^*$ but independent of $h, u$. 
The last inclusion in \eqref{good2} is due to the doubling condition of $\mu_{\hat u}$,
where $\mu_{\hat u}$ is the Monge-Amp\`ere measure of $\hat u$.

{Let $x_0=0\in\pom$, we shall}
 describe a geometric implication of \eqref{good1}. 
Denote $w=u-{\hat\ell}$ and assume that $w$ attains its minimum at $p_0$.
Let $\phi$ be a convex function whose graph is a convex cone with vertex at $(p_0, w(p_0))$ and satisfies
$\phi=w$ on $\p S^c_h(0)$. Then we have
\beq\label{incl0}
\nabla\phi(x)\cdot (x-p_0) \le \nabla w(x)\cdot (x-p_0)\le  c \nabla\phi(x)\cdot (x-p_0) \ \ \ \forall\ x\in \p S^c_h(0).
\eeq
The first inequality is due to the convexity of $w$ and the second one is due to \eqref{good1}.

Let $p\in \p S^c_h(0)$ such that
$p\cdot e_1=\sup\{ x\cdot e_1: \ x\in  S^c_h(0)\}$, 
where $e_k$ denotes the unit vector on the $x_k$-axis, for $k=1, 2, \cdots, n$.
Then $\nabla w(p)=|\nabla w(p)| e_1$ and \eqref{incl0}  implies that 
\beq\label{gesti}
|\nabla w(p)| \approx \frac{h}{(p-p_0)\cdot e_1}.
\eeq
By the convexity, one sees that \eqref{gesti} holds if $p_0$ is replaced by any point in $\frac 12 S_h^c(0)$.
In particular it holds when $p_0=0$.  \eqref{gesti} will be used in the proof of Lemma \ref{luni} below.

{In this paper we use the notation $a\gtrsim b$ (resp. $a\lesssim b$) 
if there exists a constant $C>0$ depending only on $n,f,\Om,\Om^*$ such that $a \geq C b$ (resp. $a\leq C b$), 
and $a\approx b$ means that $C^{-1} b\le a\le C b$.}
For a convex set $A$, we also use the notation $A\sim E$, where $E$ is an ellipsoid, 
if  $C^{-1} E\subset A\subset  C E$.
For two convex sets $A_1$ and $A_2$, we denote $A_1\sim A_2$ if there is an ellipsoid $E$ such that 
$A_1\sim E$ and $A_2\sim E$. 
If $A\sim B$ for a ball $B$, we also say that $A$ has a good shape.
 
The following lemma shows an equivalence relation between these two sub-level sets.

\begin{lemma}\label{rela1}
Under the hypotheses of Lemma \ref{Lemma2.2}, for $h>0$ small we have
	\begin{equation}\label{relate1}
		S_{b^{-1}h}^c(0)\cap\Omega \subset S_h(0) \subset S_{bh}^c(0)\cap\Omega,
	\end{equation}
where the constant $b\ge 1$ depends only on $n, \lambda$, $\Om, \Om^*$, but is independent of $h$ and $u$.
\end{lemma}

\begin{proof}
To prove the first inclusion, it suffices to prove that  for any $x\in S^c_h(0)$,
we have $u(x)\le Ch$ for a constant $C>0$ depending only on $n$.
Indeed, assume that $\sup \{ u(x) :  x\in S^c_h(0)\}$ is attained at $p\in \p S^c_h(0)$.
Let $q=-\beta p$, where $\beta>0$, be a point on $ \p S^c_h(0)$ such that $p,q,0$ stay on a line segment.
Since  $0$ is the centre of $S^c_h(0)$, we have $c_n^{-1}\le \beta \le c_n$ for a constant $c_n$ depending only on $n$. 
Noting that {$u(0)=0$ and} $u=\ell$ on $\partial S_{h}^c(0)$ for a linear function $\ell$, we have
$$\ell(q)+\beta \ell(p)=(1+\beta)\ell(0) =(1+\beta)h.$$
If $\ell(p)=u(p)>Ch$ for a large $C$, we have $u(q)=\ell(q)<0$, which is a contradiction.
	
The second inclusion follows readily from the strict convexity, \eqref{stricon}.
\end{proof}


The following uniform density was introduced and proved by Caffarelli in \cite{C96}, assuming 
that $\Om$ is polynomially convex. 
Here we relax the polynomial convexity to the convexity of domains with $C^{1,1}$ boundary.

\begin{lemma}\label{luni}
Assume that $\Omega$, $\Omega^*$ are bounded convex domains with $C^{1,1}$ boundary, and that $0\in\partial\Omega$.
Then
	\begin{equation}\label{unid}
		\frac{\Vol\left(\Omega\cap S^c_h(0)\right)}{\Vol\left(S^c_h(0)\right)} \geq \delta_0 >0
	\end{equation}
for some positive constant $\delta_0$ depending on $n, \lambda$, $\Om, \Om^*$, but independent of $u$ and $h$.
\end{lemma} 
 
\begin{proof}
Assume that $\{x_n=0\}$ is the tangential plane of $\pom$ at $0$ and $\Om\subset\{x_n>0\}$. 
Let $S'_h$ and $S'_{\Om, h}$ be respectively the projections of  $S^c_h$  and $S^c_h\cap \Om$ on $\{x_n=0\}$.
To prove \eqref{unid}, it suffices to prove 
\beq\label{unid1}
|S'_{\Om,h}|\ge C |S'_h| .
\eeq
{In fact, let $\tilde p=r_ne_n\in\p S_h^c$, see Figure 1 below. Then we have $\Vol(S_h^c)\leq C_1r_n|S'_h|$ and $\Vol(\Om\cap S_h^c)\geq C_2r_n|S'_{\Om,h}|$, where the constants $C_1, C_2$ only depend on the dimension $n$.}

For any unit vector $e\in \{x_n=0\}$, denote
\begin{align*}
\lambda_e & =\sup\{ (x-y)\cdot e : \ x, y\in S'_{\Om,h}\},\\
 r_e & = \sup\{t : te\in S'_h\}. 
 \end{align*} 
Note that $\lambda_e$ is the width of projection of $S'_{\Om,h}$ in the direction $e$, 
and $r_e$ is the distance from $0$ to the boundary $\p S'_h$ in  the direction $e$.
We \textit{claim} that if there is a positive constant $C$ such that
	\begin{equation}\label{newcon}
		\frac{\lambda_e}{r_e}\ge C\quad \forall\ e\in {\p B_1(0)}\cap\{x_n=0\},
	\end{equation}		
then \eqref{unid1} holds.

To prove this claim, we use induction on dimensions. 
Let $E$ be the minimum ellipsoid of $S'_h$ with principal radii $r_1\leq\cdots\leq r_{n-1}$ 
and principal axes $e_1,\cdots, e_{n-1}.$

{Let $p\in \partial S'_{\Om,h} $ be a point satisfying $|p\cdot e_{n-1}|=\sup\{{|x\cdot e_{n-1}|}: x\in S'_{\Om,h}\},$ and $e_p:=\frac{p}{|p|}$.
By \eqref{newcon}, $|p\cdot e_{n-1}|\geq Cr_{n-1}$.}
Let $S''_{\Om,h}$ be the projection of $S'_{\Om,h}$ on $\{x: x\cdot e_p=0\},$ and $S''_h:=S'_h\cap \{x:x\cdot e_p=0\}.$  
Denoting
\begin{align*}
\lambda'_e & =\sup\{ (x-y)\cdot e : \ x, y\in S''_{\Om,h}\},\\
 r'_e & = \sup\{t : te\in S''_h\}
\end{align*}
for any unit vector $e\in\text{span}(e_1, \cdots, e_{n-1})$ and $e\perp e_p$, we still have $\frac{\lambda'_e}{r'_e}\ge C$. 
Observe that 
$$|S'_{\Om,h}|\approx   |S''_{\Om,h}| |p|  \geq C|S''_{\Om,h}|r_{n-1}\ \ \  \text{and}\ \ \ 
|S'_h|\leq C|S''_h|r_{n-1}.$$
Therefore, to prove \eqref{unid1} it suffices to prove
$$|S''_{\Om,h}|\geq C|S''_h|.$$  
By induction we can reduce it to one-dimensional case, in which the claim is trivial.

Let $e_1$ be the direction in which $\inf\big\{\frac{\lambda_e}{r_e}:\  e\in \{x_n=0\} \big\}$ is attained. 
By the above claim, it suffices to prove that $\frac{\lambda_{e_1}}{r_{e_1}}\ge C$.
Let $\ell$ be the linear function such that  $u=\ell$ on $\p S^c_h(0)$.
By subtracting a linear function we assume that $\ell=0$ (namely we write $u-\ell$ as $u$).
Assume $u$ attains its minimum at $p_0$.
Let $p_l$ and $p_r$ be the left and right ends of $S^c_h$, namely
\begin{align*}
 p_r\cdot e_1=\sup\{x\cdot e_1: \ x\in S^c_h(0)\},\\
 p_l\cdot e_1=\inf\{x\cdot e_1: \ x\in S^c_h(0)\}.
\end{align*}
Denote $q_l=Du(p_l)$ and $q_r=Du(p_r)$. 
By definition, $r_{e_1}e_1\in\partial S'_h,$ and there 
exists $y\in \partial S^c_h(0)$ such that the projection of $y$ on $\{x_n=0\}$ is $r_{e_1}e_1.$ 
Since $S^c_h(0)$ is balanced with respect to $0$, we may
assume $y_n=y\cdot e_n\geq 0.$ 
Observe that $p_r\cdot e_1\geq y\cdot e_1= r_{e_1}.$

\begin{figure}[h]
 \centering
 \includegraphics[scale=0.36]{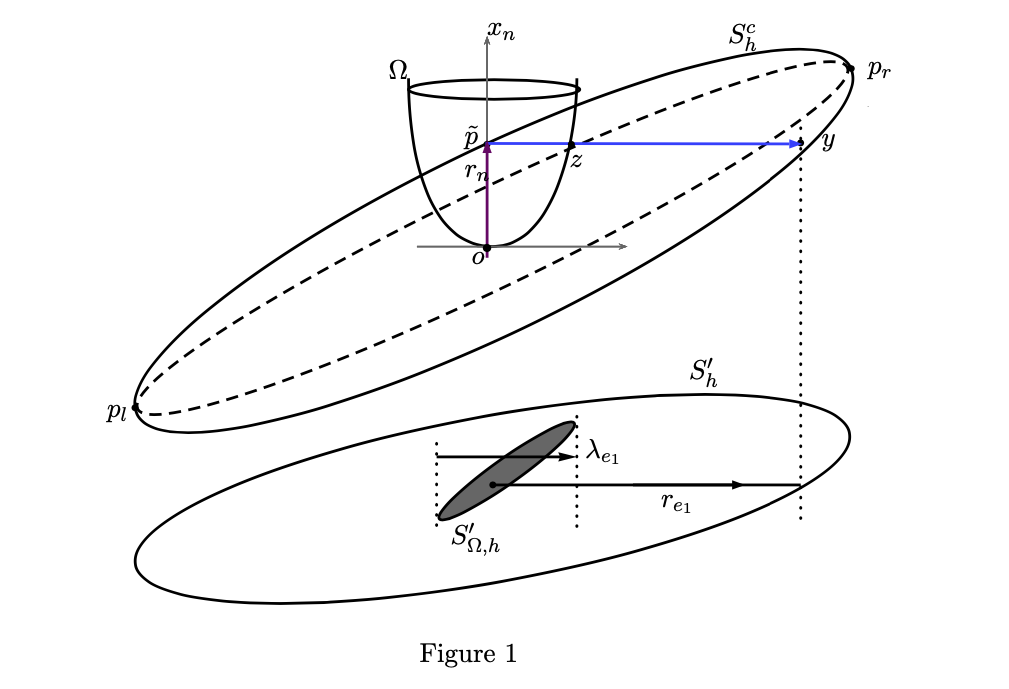}
\end{figure}

 If the ratio $\lambda_{e_1} / r_{e_1}$ is sufficiently small, we have 
\begin{itemize}
\item [a)] $\delta y\not\in\Omega$ for some small $\delta>0.$

\item [b)]   $ \delta p_r,  \delta p_l\not\in\Omega$.

\item [c)]  $q_l, q_r\in\pom^*$. 

\item [d)]   The segment $\overline{q_lq_r}$ is parallel to $e_1$. 

\item [e)] The point $q_0 := Du(p_0)$ lies on the segment $\overline{q_lq_r}$, and
by \eqref{good2},  $|q_0-q_l|\approx |q_0-q_r|$.

\item [f)]  By the convexity of $S^c_h$, there is a unique number $r_n>0$ such that $\widetilde p=:r_n e_n\in \p S^c_h$. 
The line segment $\overline{\widetilde p y}$ intersects with $\pom$ at a point $z=(z_1,\cdots, z_n)$.
Since both points $\widetilde p, y \in \p S^c_h$, we have $z\in S^c_h$.
By definition we have $ \lambda_{e_1}\ge |z'|$, where $z'=(z_1, \cdots, z_{n-1})$.
Hence by property a)  above and since $y_n\ge 0$, we infer that  
$$z_n\ge \frac {r_{e_1}-|z'|}{r_{e_1}} r_n\ge \frac 12 r_n. $$
{Actually, in the triangle vertex at $(0, \tilde p, y)$, since $z\in\overline{\tilde p y}$, one has
\begin{equation*}
\begin{split}
	z_n &\ge \frac {r_{e_1}-|z'|}{r_{e_1}} r_n \\
		&\ge  \frac {r_{e_1}-\lambda_{e_1}}{r_{e_1}} r_n \ge \frac 12 r_n,
\end{split}
\end{equation*}
as the ratio $\lambda_{e_1}/r_{e_1}$ is sufficiently small.
}

\item [g)] By the $C^{1,\delta}$ regularity of $u$, we have 
$r_n\ge C h^{\frac{1}{1+\delta}}$. 
By the $C^{1,1}$ regularity of $\pom$ and property f) above, we then have
\beqs
  r_{e_1}>|z'| \ge C z_n^{1/2}\ge  C h^{\frac{1}{2(1+\delta)}} . 
\eeqs
\end{itemize}

Let $q^*\in \pom^*$ be the point such that
$$|q_0-q^*|=\inf \{|q-q_0| : \ q\in \pom^*\} .$$
Assume that $q^*=q_0+ \sigma e^* $ for a unit vector  $e^*$.
Note that  $|p_l-p_r|$ is small if $h$ is small. 
Hence by the $C^{1,1}$ smoothness of $\pom^*$ and property e) above,
we see that 
\beq
\sigma=|q^*-q_0|\le C |q_l-q_r|^2 \ \ \text{as}\ h\to 0.
\eeq
By \eqref{gesti} (note that \eqref{gesti} holds when $p_0$ is replaced by any point in $\frac 12 S^c_h(0)$),
\beqs
|q_r-q_l|=|Du(p_l)-Du(p_r)|\le C\frac{h}{p_r\cdot e_1}\le C \frac {h}{r_{e_1}} \le C h^{\frac{1+2\delta}{2+2\delta}} .
\eeqs
Hence 
\beqs
\sigma=|q_0-q^*|\le C |q_r-q_l|^2\le C h^{\frac{1+2\delta}{1+\delta}}= C h^{1+\frac{\delta}{1+\delta}}.
\eeqs
But by \eqref{gesti} again,  we also have
\beqs
\sigma \approx \frac {h}{d_{e^*}},\quad \mbox{ where }\ \ d_{e^*} := \sup\{x\cdot e^* : x\in S_h^c(0)\}. 
\eeqs
Hence
\beqs
d_{e^*}\approx \frac h\sigma \ge h^{-\frac{\delta}{1+\delta}}\to\infty \ \ \text{as}\ h\to 0.
\eeqs
This is apparently a contradiction, because $d_{e^*}\to 0$ as $h\to 0$, by the strict convexity of the solution.
\end{proof}

\begin{remark}\label{2duni}
As mentioned before Lemma \ref{luni}, the uniform density was proved by Caffarelli \cite[Remark 2, Theorem 3.1]{C96}, assuming that $\Omega$ is polynomially convex. In dimension two, a bounded convex domain is polynomially convex.
Hence when $n=2$,
the uniform density holds for any bounded convex domains. 
No regularity on the boundaries $\pom$ and $\pom^*$ is needed.
\end{remark}

From the uniform density property, we then have \cite{C96},
	\begin{equation}\label{gV1}
		\Vol\left(S_h^c(0)\right) \approx \Vol\left(S_h^c(0) \cap \Omega \right) \approx h^{\frac{n}{2}}
	\end{equation}
for any $h>0$ small. By Lemma \ref{rela1}, we also have 
	\begin{equation}\label{gV2}
		\Vol\left(S_h(0)\right) \approx h^{\frac{n}{2}}.
	\end{equation}
The following duality result can be found in \cite [Corollary3.2]{C96}
\begin{corollary}[Duality]\label{cdual}
Let $T$ be a unimodular linear transform such that 
$B_{h^{1/2}}\subset T\{ S_h^c[u](0) \}\subset B_{nh^{1/2}}$.
Then we have
\beq \label{BTB}
B_{Ch^{1/2}}\subset T^*\{ S_h^c[v](0) \}\subset B_{C^{-1}h^{1/2}},
\eeq
where $T^*=(T')^{-1}$ is the inverse of the transpose of $T$.
\end{corollary}
 
\begin{proof} 
As the inner product $x\cdot y$ is invariant under the transforms $T$ and $T^*$,
to prove \eqref{BTB} one may assume directly that $T$ is the identity mapping. 
Then \eqref{BTB} follows from \eqref{gV1} and \eqref{good2}. 
\end{proof}

From Corollary \ref{cdual}, we also have the following corollaries, which will be used in \S\ref{sn51}.

\begin{corollary} \label{cdual-1}
For any $h>0$ small, we have 
\beq\label{inn}
|x\cdot y|\le C h \quad  \ \ \ \forall\    x\in S_h^c[u](0),\   y\in S_h^c[v](0).
\eeq
Moreover, for any $x\in \p S_h^c[u](0)$, there exists $y\in \p S_h^c[v](0)$ such that 
\beq\label{inn1}
x\cdot y\ge C^{-1} h,
\eeq
where $C$ is a constant independent of $u$ and $h$.
\end{corollary}

\begin{remark}\label{dforsh}
Similarly to Corollary \ref{cdual-1}, by Lemma \ref{rela1} we also have the following relation between $S_h[u](0)$ and $S_h[v](0):$
\begin{equation}\label{cdual2}
|x\cdot y|\le C h \quad  \ \ \ \forall\    x\in S_h[u](0),\   y\in S_h[v](0).
\end{equation}
\end{remark}

\begin{remark} \label{cwidth}
Given any unit vector $e\in\mathbb{R}^n$,
let $$d_1:=\sup\left\{|x\cdot e|: x\in S_h^c[u](0)\right\},\quad d_2:=\sup\left\{|x\cdot e|: x\in S_h[u](0)\right\}$$ be the  width of $S_h^c[u](0)$ and $S_h[u](0),$ respectively, in $e$ direction. 
Note that $S^c_{bh}(0)\subset CbS^c_h(0)$ and $S^c_h(0)\subset CbS^c_{b^{-1}h}(0)$, where $b$ is the constant in Lemma \ref{rela1} and $C$ is a constant independent of $h$, (see \cite[Observation b) in Lemma 4.1]{C96}).
Then by Lemma \ref{rela1}, Lemma \ref{luni}, \eqref{gV1} and \eqref{gV2}, we can obtain $d_1\approx d_2.$
\end{remark}

\begin{remark}\label{invlim}
The estimates in this section are invariant under affine transforms. 
Let $S_{h_j}[u](x_j)$ be a sequence of sub-level sets and
let $T_j$ be a linear transform such that $T_j\big(S_{h_j}[u](x_j)\big)$ has a good shape and $T_j(x_j)=0$,
where $x_j\in\bom$ and $h_j\to 0$ as $j\to\infty$. Denote
	\beq\label{rmaf} 
		u_j (x):=\frac{1}{h_j}u(T_{h_j}^{-1}x)\quad\text{ and }\quad \Om_j:=T_j(\Om). 
	\eeq
Then the estimates in Lemmas \ref{Lemma2.2}, \ref{rela1} and \ref{luni} also hold for $u_j, \Om_j$ 
with the same constants $r, b, \delta_0$ independent of the sequence $h_j$. 
Assume that $u_j, \Om_j$ sub-converge as $j\to\infty$ to limits $u_0, \Om_0$. 
One sees that these estimates hold for $u_0$ near $0$ as well, 
again with the same constants $r, b, \delta_0$, which depend only on $n, \lambda, \Om, \Om^*$, 
but are independent of $\Om_0$.
Similarly, by taking limits, the estimates in Corollaries \ref{cdual}--\ref{cdual-1}
for the centred sub-level sets $S_{h_j}^c[u], S_{h_j}^c[v]$ also hold for the limits $u_0, v_0$.

Furthermore, 
by Caffarelli's geometric decay estimate (see \cite[Lemma 4]{C92},  \cite[Lemma 2.2]{C96}), 
one infers the strict convexity and $C^{1,\delta}$ regularity of solutions, namely
\begin{equation}\label{cc1alpha}
C^{-1} |x|^{1+ \delta^{-1} }\le u(x)\le  C|x|^{1+\delta}\ \ \forall\ x\in S_1[u](0),
\end{equation}
if $u(0)=0, Du(0)=0$ and $S_1[u](0)$ is normalised, 
where $C,  \delta>0$ depend only on $n,\lambda$  (assuming that $\pom\cap\p S_1[u](0)$ and $\pom^*$ are convex). 
\eqref{cc1alpha} also holds for the sequence $u_j$ and its limit $u_0$ near $0$,
with the same constants.
\end{remark}

\section{Tangential $C^{1,\alpha}$ regularity}\label{s3}

The tangential $C^{1,\alpha}$ regularity of $u$, for any given $\alpha\in(0,1)$, was established in \cite {C96},
where  $\Om$ is assumed to be a uniformly convex domain with $C^2$ boundary.
But the same strategy applies to convex domains with $C^{1,1}$ boundary.
To see this let us outline the proof here. 

Let $0\in\pom$ be a boundary point. We assume that locally 
$\partial\Omega$ is given by $\{x_n=\rho(x')\}$ for some convex function $\rho\in C^{1,1}$ satisfying
\begin{align*}
  &  \rho(0)=0, \ \ \ \ D\rho(0)=0,\\ 
  &  \rho(x')\leq C|x'|^2, 
\end{align*}
where $x'=(x_1,\cdots, x_{n-1})$. 
In this section, we assume that $0<f \in C^0(\overline\Omega)$ and $f(0)=1$.
To prove the tangential $C^{1,\alpha}$, it suffices to prove

\begin{lemma}\label{treg}
For any given $\alpha\in (0, 1)$, there exists a small constant $C=C_\alpha>0$ depending only on $n$, the modulus of continuity of $f$ and $\|\partial\Om\|_{C^{1,1}}$, but independent of $h$, 
such that for the centred sub-level set  $S^c_h(0)$, we have 
	\begin{equation}\label{incl} 
		S^c_h(0)\cap\{x_n=0\} \supset B_{C_\alpha h^{{1}/{(1+\alpha)}}}(0)\cap \{x_n=0\}  .
	\end{equation}
\end{lemma}

The idea of the proof is as follows.
 For each $h>0$, there is an ellipsoid  $E_h$ such that
	\begin{equation*}
		S_h^c(0) \sim E_h=\bigg\{ \sum_{i=1}^{n-1}\Big(\frac{x_i-k_ix_n}{a_i}\Big)^2 + \Big(\frac{x_n}{a_n}\Big)^2 \leq 1 \bigg\},
	\end{equation*}
where $a_1\leq\cdots\leq a_{n-1}$, namely
$\beta E_h\subset S_h^c(0) \subset  \beta^{-1} E_h$ for some constant $\beta$ depending only on $n$.
Let $be_n$ be the intersection of the positive $x_n$-axis and $\p E_{h}$.

 We first make a linear transform 
\begin{equation}\label{normT1}
	\mathcal{T}_1:\qquad	\Big\{\begin{array}{l} 
			y_i = x_i - k_i x_n \qquad i<n , \\
			y_n = x_n .
		\end{array}
	\end{equation}
This transformation $\mathcal{T}_1$ moves the centre of $E_{h}\cap\{x_n=b\}$ to the point $be_n$. 
Hence, the ``slope" $k_i$  is bounded by 
	\begin{equation}\label{slope33}
		k_i  \le \frac{a_i}{b} \quad \mbox{ for } i=1,\cdots,n-1. 
	\end{equation}

If the inclusion \eqref{incl} does not hold, 
let $h_0>0$  be the largest constant such that \eqref{incl} holds for $h>h_0$ and 
$\p S_{h_0}^c(0)\cap\{x_n=0\}$ touches $\p B_{C_0 h_0^{{1}/{(1+\alpha)}}}$, 
where the constant $C_0$ is chosen small so that
$h_0$ is also small. Then 
\beq\label{width}
a_1\leq C_0 h_0^{{1}/{(1+\alpha)}}.
\eeq
By the $C^{1,\delta}$ regularity of $u$ \cite {C92}, we have 
\beq\label{height}
a_n\ge b\geq Ch_0^{1/(1+\delta)}.
\eeq

Next we make the linear transform 
\begin{equation}\label{normT2}
	\mathcal{T}_2:\qquad	 
			z_i = y_i/a_i \qquad i=1, \cdots, n ,
\end{equation}
such that the sub-level set $S^c_{h_0}(0)$ is ``normalised''. 
Denote $\mathcal{T}=\mathcal{T}_2\circ \mathcal{T}_1$. 
The next  lemma shows that near the origin, the $\mathcal{T}(\Omega)$ tends to be flat in $e_1$ direction as $h_0\rightarrow 0.$

\begin{lemma}
For any given $R>0$, the limit of $\mathcal{T}(\partial\Omega)\cap B_R(0)$ (as $h_0\rightarrow 0$) is flat in $e_1$ direction.
\end{lemma}

\begin{proof}
Let $p'=(h_0^\gamma,0,\cdots,0)$ be a point  on the $x_1$-axis,
where $\gamma$ is chosen so that $\frac{1}{2(1+\delta)}<\gamma<\frac{1}{1+\alpha}$.
Denote $p=(p', \rho(p'))$ and $q=\mathcal{T}(p).$
By direct computation we have
\begin{align*}
  & q_1 =\frac{1}{a_1}(h_0^\gamma-k_1\rho(p')),\\ 
  & q_i=-\frac{1}{a_i}k_i\rho(p'), \ \ \ \ i=2,\cdots, n-1,\\
  & q_n=\frac{1}{a_n}\rho(p').
\end{align*}
Note that $\rho(p')\leq Ch_0^{2\gamma}.$
By \eqref{slope33}, \eqref{width} and \eqref{height} we have 
$$k_1\rho(p')\lesssim h_0^{\frac{1}{1+\alpha}-\frac{1}{1+\delta}+2\gamma} \ll h_0^\gamma,$$
where the last inequality is due to the choice of $\gamma.$ 
Hence  $q_1\rightarrow \infty$ as $h_0\rightarrow 0.$
It is also easy to verify that $|q_i|\leq \frac{1}{b}h_0^{2\gamma}\rightarrow 0$ 
$(i=2,\cdots, n-1)$, and $ q_n\rightarrow 0$,
as $h_0\rightarrow 0$.
Note that the above computation still works if $p'=(-h_0^\gamma,0,\cdots,0).$
Therefore the  limit of $\mathcal{T}(\partial\Omega)$ (as $h_0\rightarrow 0$) contains the $x_1$ axis. 
By convexity, we see that the limit of $\mathcal{T}(\Omega)$ is independent of the $e_1$ direction.
\end{proof}

Since $\mathcal{T}\{S^c_h(0)\}$ is normalised, 
the domain  $\mathcal{T}\big(\Omega\cap S^c_{Mh}(0) \big)$  has a good shape,
where $M>1$ is chosen such that $S^c_{h}(0)\subset \frac 12 S^c_{Mh}(0)$.
By the above discussion, 
the boundary part  $\mathcal{T}\{\pom\cap S^c_{Mh}(0) \}$ becomes flat in direction $e_1$ as $h\to 0$.  As in \cite{C96}, we
denote 
$$ D_h=\{z\in\mathcal{T}(S^c_{Mh}(0)) :\ 
  z=\hat z+te_1 \mbox{ for some } \hat z\in\mathcal{T}(S^c_{Mh}(0)\cap\Om) \mbox{ and } t\in\R\}$$
by erasing the dependence on $x_1$. 
  Then
  $$ \mathcal{T}(S^c_{Mh}(0)\cap\Om)\subset D_h\subset \mathcal{T}(S^c_{Mh}(0))\cap\{x_n>0\}$$
  and near the origin, $\partial D_h$ is flat in the $x_1$-direction.
  
Let $w$ be the solution to 
\begin{equation}\label{2ndw}
\Big\{ \begin{aligned}
 &\det D^2 w  = \chi_{D_{h_0}}\ \ \text{in}\ \mathcal{T}(S^c_{Mh_0}(0)),\\
  &w  = \tilde u\ \ \text{on}\ \ \p \{\mathcal{T}(S^c_{Mh_0}(0))\} ,
  \end{aligned} 
\end{equation}
where $\tilde u(z)= |\mathcal T| ^{2/n} u(\mathcal T^{-1} (z))$. 
A key observation in this proof is that  Pogorelov's interior second derivative estimate applies to $w_{11}$, 
even though the right hand side of \eqref{2ndw} is discontinuous in $(z_2, \cdots, z_n)$,
and no regularity of $\mathcal{T}(\pom)$ in $(z_2, \cdots, z_n)$ is assumed.   
Therefore $w$ is $C^{1,1}$ in $z_1$. By the maximum principle one can give an estimate for 
$|w-\tilde u|$: 
\beq\label{wu}
|w-\tilde u|\le C[\delta_0+V_{h_0}]^{1/n} ,
\eeq
where $\delta_0=\sup\{|f(x)-1|:\  x\in S^c_{Mh_0}(0)\cap\Om\}$, 
and 
$V_{h_0}=\text{Vol}\{ D_{h_0}-\mathcal T(S^c_{Mh_0}(0)\cap\Om)\}=o(h_0).$
Changing back one obtains an estimate for $u$ from the estimate $\p^2_1w\le C$, 
from which one infers the tangential $C^{1,\alpha}$ for any given $\alpha\in (0,1)$. 
For details see \cite{C96}.

\begin{corollary}\label{newcoro}
Assume that the function $f$, defined in $\Om$, is a positive constant near the origin, and both $\pom$ and $\p\Om^*$ are flat near the origin in a direction $e$.
Then near the origin, $u$ is $C^{1,1}$ and uniformly convex in the direction $e$. 
\end{corollary}

\begin{proof}
From the assumption, one can see that $\delta_0$ and $V_{h_0}$ vanish in estimate \eqref{wu},
thus $u\in C^{1,1}$ in the direction $e$. 
Since $\pom^*$ is also flat in direction $e$, by Corollary \ref{cdual} we have $B_{Ch^{1/2}}(0)\subset S_h^c[v](0)$ along $e$ direction. 
And then by the duality in Corollary \ref{cdual-1}, we have ${S_h^c[u]}(0)\subset B_{Ch^{1/2}(0)}$ along $e$ direction. 
Hence $u$ is uniformly convex in the direction $e$. 
\end{proof}

\section{Uniform obliqueness in dimension two}\label{s4}
 
The uniform obliqueness (Lemma \ref{lem obli}) is a key ingredient in proving the  boundary $C^{2,\alpha}$ and $W^{2,p}$ estimates.
The proof is technically rather complicated. For the reader's convenience, we divide the proof into two sections.
In this section we prove Lemma \ref{lem obli} in dimension two.
In dimension two, we assume that $\Om, \Om^*$ are bounded convex domains with
$C^{1, \gamma}$ boundaries for a small $\gamma>0$, and $f\in C^0(\bom)$.
In the next section we prove Lemma \ref{lem obli} in high dimensions.
In dimension two, our proof consists of the following four steps.

\begin{itemize}
\item [(i)] 
If the uniform obliqueness does not hold at the origin, 
we express the boundaries of $\Om$ and $\Om^*$ by \eqref{rho*}, and
prove a ``balance property'' of the sub-level set $S_h[u](0)$ in Lemma \ref{lemban}. 
It implies the decay estimates \eqref{key1} and \eqref{key2}.  

\item [(ii)] 
We introduce a blow-up argument so that the inhomogeneous term $f$ becomes a positive constant in the limit. 
 
\item [(iii)]
The blow-up limit $u_0$ may not be smooth.
We construct a smooth sequence $\{u_k\}$, which converges to $u_0$ locally uniformly.  

\item  [(iv)] 
We introduce the auxiliary function  $w = \p_1u_0+u_0-x_1\p_1u_0$. 
By Steps (ii), (iii) and the maximum principle, the function $\underline{w}(t) = \inf w(t,\cdot)$ is concave near the origin.
The concavity and the decay estimate  \eqref{key2} imply that $\underline w\equiv 0$ for $t>0$ small, which contradicts to the strict convexity of $u_0$. Hence we infer the uniform obliqueness.
\end{itemize}

\subsection{Balance property and decay estimate}\label{ss41}

Assume that $0\in\pom$ and $\Omega\subset\{x_2>0\}$. 
To prove the uniform obliqueness, 
by {the global $C^{1,\delta}$ regularity \cite {C92}},
we may assume to the contrary that $u(0)=0$, $Du(0)=0\in \partial\Omega^*$ and $\Omega^*\subset\{y_1> 0\}$. 
Then we have
\begin{enumerate}[$(i)$]
\item $u_1=: u_{x_1}>0$ in $\Omega$ and $v_2=: v_{y_2}>0$ in $\Om^*$;  it implies that  \label{ppi}
\item if $x\in S_h(0)$, then $x-te_1\in S_h(0)\ $ $\forall\ t>0$, provided $x-te_1\in\Omega$, \label{ppii}
\end{enumerate}
where $S_h(0)=S_h[u](0)$ is the sub-level set of $u$, introduced in \S\ref{s2}. 
Accordingly,  the boundaries $\partial\Omega$ and $\partial\Omega^*$ near the origin can be expressed as
\beq\label{rho*}
 {\begin{split}
		\partial\Omega &= \{x_2 = \rho(x_1)\}, \\
		\partial\Omega^* &= \{y_1 = \rho^*(y_2)\},
\end{split} } 
\eeq
with the following properties:
\begin{itemize}

\item [($\text{\bf H}_1$)]  $\rho, \rho^* \ge 0$ are convex functions defined in an interval  $(-r_0, r_0)$ 
and satisfying $\rho(0)=0$ and $\rho^*(0)=0$, where $r_0>0$ is a constant.

\item [($\text{\bf H}_2$)]  Denote  $\sigma(t) = |t|^{1+\gamma}$. 
By the assumption $\pom\in C^{1, \gamma}$, we have
\begin{equation}\label{bdrya}
		\rho(t) \leq C\sigma(t) \quad \mbox{ for } t\le 0.
	\end{equation}
 \end{itemize}

\begin{remark}\label{R4.1}
(i)
We will derive a contradiction from ($\text{\bf H}_1$) and ($\text{\bf H}_2$).
Note that it suffices to assume $t\le 0$ in \eqref{bdrya}. 
By Lemma \ref{lemban} below, we can show that \eqref{bdrya} holds for $t>0$ as well. 
Hence $\pom$ is $C^1$ at $0$. 
\\
(ii)
As the reader will see, in our argument below we will not use any boundary regularity for $\pom^*$.
For any point $p\in \pom\cap\{x_1<0\}$, 
since the inner product $\langle \nu(p), \nu^*(Du(p))\rangle\ge 0$, $(\text{\bf H}_2)$ implies that 
$\rho^*(t)=o(t)\ \text{for}\ t\ge 0. $ 
\\
(iii)  For clarity in this section we will always assume that $\pom, \pom^*\in C^{1, \gamma}$, for a small $\gamma>0$.
By Remark \ref{invlim}, the constants in this section depend on $n, \lambda, \Om, \Om^*$ (inner and outer radii of $\Om,\Om^*$ and $\gamma$). In the approximation $\{u_k\}$ in \S\ref{ns43}, we also allow that the constants depend on $k$. But all the constants are independent of $h$ and $u$ (for $h>0$ small). The continuity of $f$ is used only in the blow-up process, such that the RHS of \eqref{eq00} is a constant. 
In this section we do not use the tangential $C^{1,\alpha}$ regularity of \S\ref{s3}.
\end{remark}

Let $q=(q_1,q_2)$ and $\xi=(\xi_1,\xi_2)$ be two points  on $\partial S_h(0)\cap \overline\Omega$ such that
	\beq\label{qxi}
	 {\begin{aligned}
		&\langle q, e_1 \rangle = \sup\{\langle x, e_1 \rangle : x\in S_h(0)\},\\
		&\langle \xi, e_1 \rangle = \inf\{\langle x, e_1 \rangle : x\in S_h(0)\}.
		\end{aligned}}
	\eeq
Apparently  $q_1>0$ and $\xi_1<0$, see Figure 2 below.
Note that $u_{x_2}(p)<0$ for any boundary point $p\in\pom\cap\{x_1>0\}$. Hence $q$ is an interior point of $\Om$.
The following lemma shows that the area of $S_h[u](0)\cap\{x_1>0\}$ can balance that of $S_h[u](0)\cap\{x_1<0\}$.

\begin{lemma}\label{lemban}
For all $h>0$ small, we have the ``balance" property
	\begin{equation}\label{rbalan}
		q_1 \geq \delta_0|\xi_1|,
	\end{equation}
where $\delta_0>0$ is a constant independent of $h$. 
\end{lemma}

\begin{figure}[h]
\centering
\includegraphics[scale=0.33]{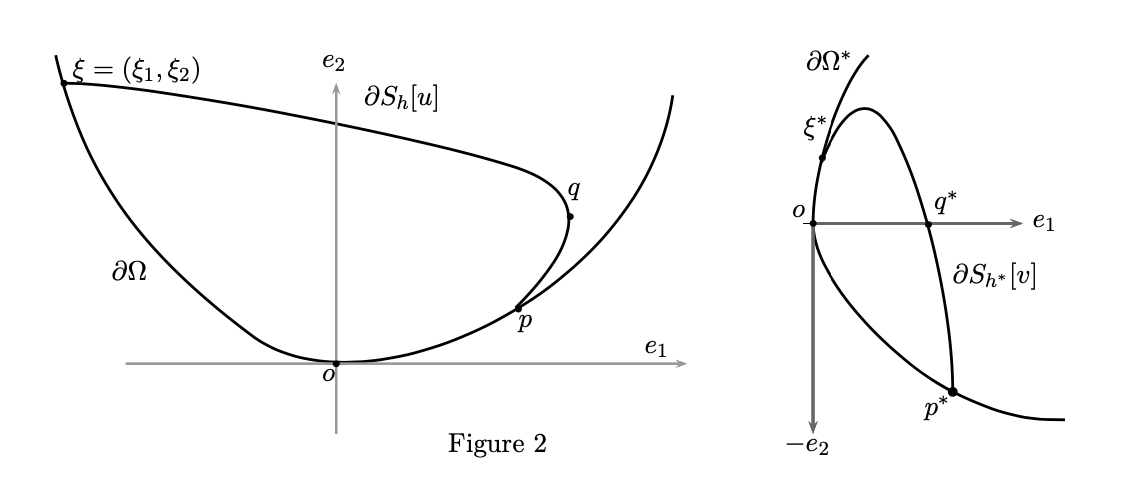}
\end{figure}

\begin{proof}
To prove \eqref{rbalan}, 
suppose to the contrary that
	\begin{equation}\label{contra}
		q_1 = o(|\xi_1|)
	\end{equation}
for a sequence $h\to 0$. 
Denote $t_0=\frac12(\xi_1+q_1)$. 
There is a unique $s_0>\rho(t_0)$ such that $u(t_0,s_0)=h$.
Denote
	\begin{equation*}
		x^c=(x^c_1, x^c_2) := \big( t_0, {\Small\text{$\frac12$}}  (s_0+\rho(t_0) ) \big).
	\end{equation*}
The point $x^c$ can be regarded as the centre of $S_h(0)$.
Denote 
	\begin{equation*}
	\Big\{ \begin{aligned}
		\lambda_1 &= q_1-\xi_1, \\[-3pt]
		\lambda_2 &= s_0-\rho(t_0).
	\end{aligned} 
	\end{equation*}
Apparently
	\begin{equation}\label{gooda}
		\Vol(S_h(0)) \approx \lambda_1\lambda_2.
	\end{equation}
Moreover, by \eqref{bdrya} and the property $(i)$, we have
	\begin{equation}\label{goodb}
		\lambda_2 \leq \xi_2 \leq \sigma(\xi_1)  
			\leq \sigma(\lambda_1).
	\end{equation}

By \eqref{contra}, we have $\frac12\xi_1< x^c_1 < \frac14\xi_1$ for $h>0$ small. 
Let's make the first change
	\begin{equation}\label{chan1}
	\begin{split}
		y_1 &= x_1, \\
		y_2 &= x_2 - \frac{x^c_2}{x^c_1}x_1
	\end{split}
	\end{equation}
such that $S_h[u]\subset\{x\in\mathbb{R}^2 : \xi_1<x_1<q_1, |x_2|<4\lambda_2\}$.
Note that such a change does not change the ratio $\frac{q_1}{|\xi_1|}$ in \eqref{rbalan}.
We then make the change
	\begin{equation}\label{chan2}
	\begin{split}
		z_1 &= y_1/\lambda_1, \\
		z_2 &= y_2/\lambda_2,
	\end{split}
	\end{equation}
and accordingly,
	\begin{equation}\label{chan3}
		u_h= u/h, \quad \mbox{where } h=u(\xi).
	\end{equation}
By the volume estimate \eqref{gooda}, the sub-level set $S_h(0)$ has a  good shape after changes \eqref{chan1} and \eqref{chan2}.
By Lemmas \ref{rela1} and \ref{luni} 
(in dimension two, the uniform density holds for any bounded convex domains, see Remark \ref{2duni}),
the centred sub-level set $S^c_h(0)$ also has a good shape after the change.
Hence,  
\beq\label{ulc}
 \|u_h\|_{L^\infty(B_1(0))}\leq C
\eeq
for a constant $C$ independent of $u$ and $h$.

Let $q_h, \xi_h$ be the corresponding points of $q, \xi$ after the above changes.
Assume that $q_h\rightarrow q_0=(q_{0,1}, q_{0,2})$, $u_h\rightarrow u_0$ as $h\rightarrow 0.$ 
By \eqref{contra} we have $q_{0,1}=0$, namely $q_0$ is on the $x_2$-axis. 
On the other hand, after the change \eqref{chan2},
the line $\{x_1=q_{h,1}\}$ is tangent to $S_1[u_h]$ at $q_{h}$, 
and thus $u_h(x)\geq 1$ for all $x\in\{x_1\geq q_{h,1}\}.$ 
Passing to the limit we have $u_0(x)\geq 1$  for $x\in \{x_1\geq 0\}$ 
and $u_0(0)=\lim_{h\to 0}u_h(0)=0$,  which is a contradiction by \eqref{ulc},
namely  $u_0$ is a limit of a sequence of locally uniformly bounded convex functions $u_h$, 
$u_0$ must be continuous. 
\end{proof}

\begin{corollary}\label{co41}
For $t>0$, denote
	\begin{equation*}
		\underline{u}(t) = \inf\{ u(t,x_2) : x_2 \geq \rho(t)\}.
	\end{equation*}
We have the asymptotic estimate
	\begin{equation}\label{key1}
		\underline{u}(t) \leq  Ct\sigma(t) \ \ \ \text{for $t>0$ small}.
	\end{equation}	
\end{corollary}

\begin{proof}
By the strict convexity of $u$, the sub-level set $S_h(0)$ shrinks to the origin as $h\to0$.
Hence, for any $t>0$, there exists a unique $h>0$ such that $\{x_1=t\}$ is tangential to $\partial S_h(0)$ at the point $q=(q_1, q_2)$. 
This implies that $q_1=t$ and $\underline{u}(t)=h$.

From \eqref{rbalan}, 
	\begin{equation*}
		\lambda_1 = q_1 - \xi_1 \le C t, 
	\end{equation*}
for some constant $C$ independent of $h$. Then from \eqref{goodb} and \eqref{gooda} we have
	\begin{equation*}
		\lambda_2  \le C \sigma(t) \quad\mbox{and}\quad  \Vol(S_h(0)) \leq Ct\sigma(t).
	\end{equation*}	
Hence by \eqref{gV2} we obtain that
	\begin{equation*}
		\underline{u}(t) = h \leq  Ct\sigma(t).
	\end{equation*}
\
\vskip-36pt
\end{proof} 

\begin{corollary}
For $t>0$, denote
	\begin{equation*}
		\underline{\p_1u}(t) = \inf\{\p_1u(t,x_2) : x_2\geq\rho(t)\}.
	\end{equation*}
Then we have the asymptotic behaviour for $t>0$ small,
	\begin{equation}\label{key2}
		\underline{\p_1u}(t) \leq C\sigma(t).
	\end{equation}
\end{corollary}

\begin{proof}
This is a direct consequence of \eqref{key1}. 
In fact, by the convexity of $u$, for $t>0$ small
	\begin{equation*} 
		\p_1u(t, x_2) \leq \frac{u(2t, x_2) - u(t, x_2)}{t} \leq \frac{u(2t, x_2)}{t}.
	\end{equation*} 
Then taking the infimum in $x_2$, from \eqref{key1}, we obtain that
	\begin{equation*}
		\underline{\p_1u}(t) = \inf_{x_2} \p_1u(t, x_2) \leq \frac{\underline{u}(2t)}{t} \leq C\sigma(t). 
	\end{equation*}
\
\vskip-36pt
\end{proof}

\subsection{A blow-up sequence}\label{nss42}
Assume that $f>0$ is continuous.
The purpose of blow-up is such that $f$ becomes a positive constant in the limit.

From the proof of Lemma \ref{lemban},
the sub-level set $S_h[u](0)$ has a good shape under the following normalisation $\mathcal{T}$:
	\begin{equation}\label{norm1}
	\begin{split}
		y_1 &= x_1/\lambda_1, \quad \mbox{ with } \lambda_1 =  q_1-\xi_1, \\
		y_2 &= x_2/\lambda_2, \quad \mbox{ with } \lambda_2 = \rho(\xi_1).
	\end{split}
	\end{equation}
In fact, as shown in the proof of  Lemma \ref{lemban}, 
we have $\Vol(S_h[u])\approx \lambda_1\lambda_2$.
Hence $\mathcal{T}(S_h[u])\approx 1$. 
Also by the proof of  Lemma \ref{lemban}, 
$\mathcal{T}(S_h[u]) \subset [-1,1]\times[0,1]$. Hence $\mathcal{T}(S_h[u])$ has a good shape. 
	
Accordingly we make the change $u\to u_h$,  where
	\begin{equation}\label{norm2}
		u_h(x)=u(\mathcal{T}^{-1}x)/h. 
	\end{equation}
After the change, the domain $\Om$ is changed to $\Om_h$, 
and the boundary $\{x_2=\rho(x_1)\}$ is changed to $\big\{x_2=\rho_h(x_1)=\frac{1}{\lambda_2}\rho(\lambda_1x_1)\big\}$.
	By Lemma \ref{rela1} and Lemma \ref{luni} we also have 
	\begin{equation}\label{cengood}
	B_{\frac{1}{C}}(0)\subset \mathcal{T}(S^c_h[u](0))\subset B_C(0)
	\end{equation}
	 for some constant $C$ depending only on $n$, the constants $b$ in Lemma \ref{rela1} and $\delta_0$ in Lemma \ref{luni}, but independent of $h.$
In \eqref{cengood} the centred sub-level set $S^c_h$ can be replaced by the usual sub-level set $S_h$
if the centre of the concentric ball is properly chosen.

By \eqref{key1}, the limit $\lim_{t\to0}\frac{\underline{u}(t)}{t\sigma(t)} <\infty$.
Hence for any fixed small $\bar\eps > 0$ (we may fix $\bar\eps=1$), there is a sequence $t_j\to 0$ ($t_j>0$) such that 
	\begin{equation}\label{subs}
		\frac{\underline{u}(t)}{t\sigma(t)} \leq
		  (1+\bar\eps) \frac {\underline{u}(t_j)} {t_j\sigma(t_j)} \quad \forall\ t\in(0, t_j) . 
	\end{equation}
Denote $h_j=\underline{u}(t_j)$.
Since $\mathcal{T}(S_h[u])$ has a good shape,
for any $R>0$, $\Om_{h_j}\cap B_R(0)$ converges in Hausdorff distance to a limit. 
Hence by passing to a subsequence, 
$\Omega_{h_j}$ converges to a limit $ \Omega_0$ as $h_j\to 0$,
which is an unbounded convex domain in $\R^2$.

Next we show that $u_{h_j}$ sub-converges to a limit $u_0$ as $h_j\to 0$.
Indeed, by the geometric decay of sections (\cite[Lemma 2.2]{C96}), for any $k>0$, 
there exists a constant $M_k$ such that 
$$kS^c_h[u]\subset S^c_{M_kh}[u] \quad \text{ for $h>0$ small}. $$
Hence by the convexity of $u_{h_j}$ and the estimate \eqref{cengood},
$u_{h_j}$ is locally uniformly bounded, which implies the sub-convergence.

By the weak convergence of the Monge-Amp\`ere operator,
$u_0$ satisfies the equation
	\begin{equation}\label{eq00}
		\det\,D^2u_0 = c_1 \chi_{_{\Omega_0}} \quad\text{ in } \mathbb{R}^2 
	\end{equation}
for a positive constant $c_1$. There is no loss of {generality} in assuming that $c_1=1$.

By the change \eqref{norm1}, we have $0\in \pom_0$.
Let $\mathcal J$ be the projection of $\Om_0$ on the $x_1$-axis.
By Lemma \ref{lemban}, there is an interval $(0, r_0)\subset \mathcal J$.
Hence the {\it lower boundary} of $\Om_0\cap\{0<x_1<r_0\}$ 
can be represented by a convex function 
\beq \label{rho0}
x_2=\rho_0(x_1) 
\eeq
and   $\rho_0$ is the limit of $\rho_{h_j}$, passing to a subsequence if necessary.

\begin{remark}\label{r42}  
(i).
In Lemma \ref{lemban}, we proved that $|\xi_1|\le Cq_1$, 
but the possibility $|\xi_1|=o(q_1)$ as $h\to 0$ has not been ruled out.
{Hence, even when $\pom \in C^{1,1}$, the limit $\Om_0$ may be contained in the first quadrant, }
i.e. $\Om_0\subset\{x_1>0, x_2>0\}$. 
In this case, $\rho_0$ is defined in $\{x_1>0\}$. \\[3pt]
(ii).
No matter whether $\Om_0$ is contained in the first quadrant, we point out that 
the whole positive $x_2$ axis is contained in  $\bom_0$.
{To see this, notice that there exists a constant $\beta_0$ such that 
$\beta e_2\in\Omega$  $\forall\,\beta\in(0,\beta_0)$. 
By the transform $\mathcal{T}$ in \eqref{norm1}, 
we have $\beta e_2\in\Omega_h$ $\forall\,\beta\in(0,\frac{\beta_0}{\lambda_2})$. 
By the strict convexity of $u$, $\lambda_2\to0$ and $\beta_0/\lambda_2\to\infty$ as $h\to0$. 
Hence for any $R>0$, $Re_2\in\Omega_h$ provided $h>0$ is small enough. 
Passing to the limit, we have $Re_2\in\overline{\Om_0}$. }\\[3pt]
(iii). 
In Corollary \ref{c403} below, we will show that 
$\rho_0(t) \le C \sigma(t)$  for $t>0$ small. 
But $\rho_0$ may not be smooth.
In comparison, if $\pom\in C^2$ and is uniformly convex, 
then $\rho_0$ is a quadratic polynomial \cite{C96}.  
The lack of smoothness of $\rho_0$ in our case makes the problem much more complicated.
\end{remark}

By our choice of the sequences $t_j$ and $h_j=\underline{u}(t_j)$ in \eqref{subs}, 
the asymptotic behaviour \eqref{key1}  holds  for the limits $u_0$ and $\rho_0$.
Namely, we have the following estimates.

\begin{corollary}\label{u0est}
Denote 
$$\underline{u_0}(t)  = \inf\{ u_0(t,x_2) : x_2 \geq \rho_0(t)\}\ \ t>0.$$
We have 
	\begin{equation}\label{key5}
		\underline{u_0}(t)   \leq  Ct\sigma(t) \ \ \ \text{for $t>0$ small}.
	\end{equation}
\end{corollary}

\begin{proof}
Let $q, \xi$ be the points defined in \eqref{qxi} with $h=h_j$, 
and let $\lambda_1=q_1-\xi_1, \lambda_2=\rho(\xi_1)$ as in \eqref{norm1}.
 From \eqref{rbalan}, $t_j=q_1\approx \lambda_1.$
 
By \eqref{norm2}, $u_{h_j}(x_1, x_2)=\frac{u(\lambda_1x_1, \lambda_2x_2)}{\underline{u}(t_j)}.$
Hence, by \eqref{subs} 
\begin{align*}
\inf\{u_{h_j}(t, x_2): x_2\geq \rho_{h_j}(t)\}
&= \frac{\inf\{u(\lambda_1t, \lambda_2x_2): x_2\geq \rho(t)\}}{\underline{u}(t_j)}\\
&= \frac{\underline{u}(\lambda_1t)}{\underline{u}(t_j)} \\
 & \leq (1+\bar\eps) \frac{(\lambda_1t)\sigma(\lambda_1t)}{t_j\sigma(t_j)} \\
 & \leq Ct\sigma(t),
\end{align*}
where the constant $C>0$ is independent of $j$.  The above inequality implies that $\underline{u_{h_j}}(t) \leq Ct\sigma(t)$.
Passing to the limit, we obtain \eqref{key5}.
\end{proof}

\begin{corollary}\label{c403} 
We have 
\beq\label{n401}
\rho_0(t) \le C \sigma(t)   \ \ \ \text{for $t>0$ small}.
\eeq\end{corollary}

\begin{proof}
For any given $h>0$ small, as in \eqref{qxi} we introduce two points $\xi$ and $q$ for the sub-level set $S_h[u_0]$.
Let $z=\beta e_2$ be the point on the $x_2$-axis such that $u_0(z)=h$.
By Corollary \ref{u0est},
 $$ h=u(q) \leq Cq_1\sigma(q_1). $$
By \eqref{gV2} and Remark \ref{invlim} we have 
$$ \frac12\beta q_1 \le \Vol(S_h[u_0]) \leq Ch. $$
Hence $\beta\le C \sigma(q_1)$. Noting that $\p_{x_1}u_0\ge 0$, 
we infer that $q_2\le \beta\le C\sigma(q_1)$, and so \eqref{n401} follows.
\end{proof}

Denote $\mathcal T^*=\frac{1}{h}(\mathcal T')^{-1}$,
the dual  affine transform for $v$, where $\mathcal T'$ is the transpose of $\mathcal T$ in \eqref{norm1}.
As in \eqref{norm1}, we denote $v_h(y)=\frac 1h v((\mathcal T^*)^{-1}y)$,  
and $\Omega^*_h=\mathcal T^*(\Om^*)$. 
The boundary $\pom^*_h$ near the origin  is given by $\{y_1=\rho^*_h(y_2)\}$.
Similarly,  $\Omega^*_{h_j}$ converges to an unbounded convex domain $\Omega_0^*$ 
and $v_{h_j}$ converges to a convex function $v_0$, locally uniformly.
Moreover, 	
\begin{equation}\label{eq000}
		\det\,D^2v_0 = c_2 \chi_{_{\Omega_0^*}} \quad\text{ in } \mathbb{R}^2 
	\end{equation}
	for a positive constant $c_2.$

\begin{remark}\label{r401}  
As pointed out  in Remark \ref {r42}, 
we need to deal with the case when $\Om^*_0$ is contained in the  first quadrant
(note that by Lemma \ref{lemban}, it is again the first quadrant, not the fourth quadrant).
By Remark \ref{r42}, the whole positive $y_1$-axis is contained in $\overline{\Om^*_0}$. 
 
Although $\Om_0$ is unbounded, by Remark \ref{invlim}, $u_0$ is locally strictly convex and $u_0\in C^{1,\delta}_{loc}(\overline\Om_0)$.
By the convexity of  $\Om_0$ and $\Om^*_0$, we have $u_0\in C^1(\R^2)$. 
 By the interior regularity for the Monge-Amp\`ere equation \eqref{eq00},
$u_0$ is $C^\infty$ smooth inside $\Om_0$.
\end{remark}

\subsection{Smooth approximation}\label{ns43}

In this subsection we shall construct a smooth approximation sequences $\{u_k\}$ converging to $u_0$ in a small neighbourhood of the origin. The $C^{2,\alpha}$-smoothness of $u_k$ is needed for deriving the contradiction in \S\ref{ns44}. 
Note that we just need the $C^{2,\alpha}$-smoothness of $u_k$ but not a uniform upper bound for the $C^{2,\alpha}$-norm of $u_k$. 

Let $V=B_r\cap \Omega_0^*$ for a small constant $r>0$, and let $U=Dv_0(V)$.
By the local strict convexity of $v_0$, 
there exists $r_0>0$ such that $B_{r_0}\cap\Omega_0\subset U$.
Let $U_k, U_k^*$ be $C^\infty$ smooth, bounded domains such that
$$ {\begin{split}
  a)\ \ & 0\in \p U_k, \ \  0\in \p U^*_k \ \ \forall\ k\ge 1,  \\
  b)\ \  & U_k\subset \{x_2>0\}, \ \ U^*_k\subset \{y_1>0\} \ \ \forall\ k\ge 1, \\
  c)\ \  & U_k \to U, \ \   U^*_k \to V, \ \ \text{as $k\to\infty$, in Hausdorff's sense}.
  \end{split}} $$
Moreover, we assume the lower part of the boundary $\p U_k\cap B_{r_0}$ is the graph of a smooth, uniformly convex function $\rho_k$ in direction $e_2$, that is
\beq\label{Gak}
\Gamma_k=: \{ x_2=\rho_k(x_1)\} \cap B_{r_0}, 
\eeq
where $\rho_k$ is defined on $\mathcal{J}_k$ that is the projection of $\p U_k\cap B_{r_0}$ on the $x_1$-axis. 
In our construction, $0\in\mathcal{J}_k$ is an interior point of $\mathcal{J}_k$ for all $k\geq1$. 
By Corollary \ref{c403}, $(0,\frac12r_0)\subset \mathcal{J}_k$. 
From the above conditions $a)-c)$ and uniform convexity, the function $\rho_k$ satisfies
\beq\label{rhok1}
{\begin{split}
 & \rho_k(0) = 0; \ \ (\rho_k)'(0) = 0; \\ 
 & (\rho_k)'(x_1) > 0 \ \ \text{for}\  x_1>0; \quad (\rho_k)'(x_1) < 0 \ \ \text{for}\  x_1< 0.
 \end{split}} 
\eeq 
Meanwhile, the left part of the boundary $\p U^*_k\cap B_{r_0}$ is the graph of a smooth, uniformly convex function $\rho^*_k$ in direction $e_1$, that is
\beq\label{Ga*k}
\Gamma^*_k=: \{ y_1=\rho^*_k(y_2)\} \cap B_{r_0},
\eeq
where $\rho^*_k$ is defined on $\mathcal{J}^*_k$ that is the projection of $\p U^*_k\cap B_{r_0}$ on the $y_2$-axis. 
In our construction, $0\in\mathcal{J}^*_k$ is an interior point of $\mathcal{J}^*_k$ for all $k\geq1$. The function $\rho^*_k$ satisfies
\beq\label{rhok111}
{\begin{split}
 & \rho^*_k(0) = 0; \ \ (\rho^*_k)'(0) = 0; \\ 
 & (\rho^*_k)'(y_2) > 0 \ \ \text{for}\  y_2>0; \quad (\rho^*_k)' (y_2) < 0\ \ \text{for}\  y_2<0.
 \end{split}} 
\eeq

Let $u_k$  be the potential function for the optimal transport  from $(U_k, 1)$ to $(U^*_k, g_k)$, 
where the density $g_k=\frac{|U_k|}{|U^*_k|}$ is a constant. 
Subtracting a constant we have $u_k(0)=0$.
Since $U^*_k$ is convex, we can extend  $u_k$ to $\mathbb{R}^2$ by
\beq\label{uke}
u_k(x):=\sup\{\ell(x) :\ \ell \text{ is affine, } \ell\leq u_k\ \text{in}\ U_k,\ \nabla\ell\in U_k^*\} \quad \mbox{for }x\in\mathbb{R}^2.
\eeq
Since $u_0(0)=0$, by the uniqueness of potential functions we have 
$u_k \to u_0$ uniformly in $B_{r_0}(0)$ for a different $r_0>0$ small, 
(in this subsection, the constant $r_0$ may change from line to line but they have a uniform positive lower bound independent of $k$).
By the uniqueness of optimal mapping $Du_0$, we also have $Du_k\to Du_0$ in $B_{r_0}(0)$.

By \eqref{cc1alpha} in Remark \ref{invlim}, $u_k$ are uniformly $C^{1,\delta}$ smooth in $\overline{U_k}\cap B_{r_0}(0)$,
and the extended $u_k\in C^1(B_{r_0}(0))$, by the convexity of $U^*_k$.
Let $v_k$ be the dual of $u_k$, namely the potential function for the optimal transport  from  $(U^*_k, g_k)$ to  $(U_k, 1)$.
Similarly to \eqref{uke}, let 
\beq\label{vke}
\hat v_k(y):=\sup\{\ell(y) :\ \ell \text{ is affine, } \ell\leq v_k\ \text{in}\ Du_k(B_{r_0}(0)),\ \nabla\ell\in B_{r_0}(0)\} \quad \mbox{for }y\in\mathbb{R}^2.
\eeq
Then  $\hat v_k=v_k$ in $\overline{U^*_k}\cap B_{r_0}(0)$, $\hat v_k$ is uniformly $C^{1,\delta}$ smooth in $\overline{U^*_k}\cap B_{r_0}(0)$, 
and $\hat v_k\in C^1(B_{r_0}(0))$, for a different $r_0$.

We have constructed the sequences of functions $\rho_k, \rho^*_k$ defined on $\mathcal J_k, \mathcal J^*_k$, respectively,
with the properties \eqref{rhok1} and  \eqref{rhok111} . Moreover,  $[0, r_0/2]\subset \mathcal J_k, \mathcal J^*_k$.
The next lemma shows that $u_k$ satisfies the obliqueness condition and is smooth on $\Gamma_k$ for all $k$.

\begin{lemma}\label{2dapp} 
\begin{itemize}
\item[$(i)$] For each $k\geq1$, we have
\beq\label{n408}
\nu_{k}(x)\cdot \nu_k^*(Du_k(x)) > 0\ \ \forall\,x\in\Gamma_k,
\eeq
where $\nu_k$ and $\nu_k^*$ are the unit inner normals of the domains 
$\{x_2 > \rho_k(x_1)\}\cap B_{r_0}$ and $\{y_1 >\rho^*_k(y_2)\}\cap B_{r_0}$, respectively. 

\item[$(ii)$] For each $k\geq1$, 
$u_k$ is smooth, locally uniformly convex, and $ \det\, D^2 u_k$ is a positive constant
 in $B_{r_0}(0)\cap\{ x_2\geq \rho_k(x_1)\}$ (up to the boundary $\Gamma_k$). 
\end{itemize}
\end{lemma} 

\begin{proof}
The proof is as follows (for a fixed $k$):
\begin{itemize}
\item [$(a)$]
$(ii)$ actually follows from $(i)$. In fact, if the obliqueness \eqref{n408} holds, we can apply the argument in \S\ref{s555} to obtain the smoothness in $(ii)$.   
So, it suffices to prove $(i)$.

\item [$(b)$] 
Suppose to the contrary that the obliqueness fails for $u_k$ at a point $p_0=(t_0, \rho_k(t_0))\in\Gamma_k$ with $t_0<r_0$.
We \emph{claim} that there is an interval $(\bar t, \bar t +\eps)$, where $\bar t\ge t_0$, 
such that the obliqueness fails for $u_k$ at $\bar p=:(\bar t, \rho_k(\bar t))$, 
but holds at all points on $\Gamma_k\cap\{\bar t<x_1<\bar t+\eps\}$ (Lemma \ref{1p}). 
Then by the regularity in \S\ref{s555}, $u_k$ is  smooth up to the boundary $\Gamma_k\cap\{\bar t<x_1<\bar t +\eps\}$.

\item [$(c)$] Regard $\bar p $ as the origin.
Applying the argument in \S\ref{ns44} to $u_k$ on  $\Gamma_k\cap\{\bar t<x_1<\bar t+\eps\}$, 
we reach a contradiction at the point $\bar p$.
It implies that the obliqueness must hold for $u_k$ at all points on $\Gamma_k$, namely $(i)$ is proved.
\end{itemize} 

To apply the argument in \S\ref{ns44} to $u_k$, 
we do not need to carry out the blow-up for $u_k$ in \S\ref{nss42}, 
as $\det\,D^2 u_k$ is already a positive constant. 
We also point out that
the $C^{2,\alpha}$ regularity proved in \S\ref{s555} is localised. 
That is, let $x_0\in\pom$ and $y_0=Du(x_0)\in\pom^*$. 
If $\pom$  and $\pom^*$ are $C^{1,1}$ and convex near $x_0$ and $y_0$, respectively,
then $u$ is $C^{2,\alpha}$ smooth near $x_0$. 

Therefore it remains to verify the claim in $(b)$, which will be proved in Lemma \ref{1p} below.
\end{proof}

\begin{remark}\label{2dCaff} 
Instead of using Lemma \ref{1p} below and the regularity in \S\ref{s555}, we can also use Caffarelli's localised boundary $C^{2,\alpha}$ regularity \cite{C96} to conclude that $u_k$ is smooth up to $\Gamma_k$. 
However, the proof in \cite{C96} is rather involved.
Here we use Lemma \ref{1p} to make our proof self-contained. This lemma will also be used in \S\ref{ss42} for high dimensions. 
\end{remark}

\begin{lemma}\label{1p}
Assume that the obliqueness fails for $u_k$ at a point $p_0=(t_0, \rho_k(t_0))\in\Gamma_k$ (for any fixed $k\geq1$).
Then there is a boundary point $\bar p=(\bar t, \rho_k(\bar  t))\in \Gamma_k$ with $\bar t\geq t_0$,
such that the obliqueness fails at $\bar  p$, 
but holds at all points $p\in\Gamma_k\cap\{\bar  t<x_1<\bar  t+\eps\}$
for a constant $\eps>0$. 
\end{lemma}

\begin{proof} 
Since the obliqueness fails at $p_0$, by a change of coordinates and subtracting a linear function to $u_k$, 
we can assume that $p_0=0$, $Du_k(0)=0$, and also
	$$0\in \p U_k, \quad U_k\subset\{x_2>0\}; \quad 0\in\p U^*_k,\quad U^*_k\subset\{y_1>0\}.$$
We still use $\rho_k, \rho^*_k$ to denote the boundary $\p U_k, \p U^*_k$ near $0$, which are smooth, uniformly convex near $0$.
Correspondingly, we have  $\rho_k(0)=0, \rho_k\ge 0$, and $\rho^*_k(0)=0, \rho^*_k\ge 0$.

For any boundary point $p=(t, \rho_k(t))\in \Gamma_k$, 
let $\zeta=\zeta(p)$ denote the unit tangential vector of $\Gamma_k$ at $p$.
Denote by $\tau_1$ and $\eta$ the tangential and normal vectors of $\partial S_h[u_k]$ at $p$, respectively,
where $h=u_k(p)$.
Let $\alpha$ be the angle between $\eta$ and $\nu,$ and let $\beta$ be the angle between $\nu^*$ and $\tau_1$  
(see Figure 3).
Then, 
\beq\label{n402}
	\nu(p)\cdot \nu^*(Du_k(p))=\cos\big({\Small\text{$ \frac{\pi}{2} $}} -\alpha-\beta\big)\ge {\Small\text{$ \frac 12 $}} ( \alpha+\beta) .
\eeq
By definition, the obliqueness holds at $p$ if and only if $\nu(p)\cdot \nu^*(Du_k(p))>0$.

We claim that $\beta>0$. Indeed,
since $\rho^*_k$ is smooth,  $(\rho^*_k)'(0)= 0$.
Noticing that  $p^*=Du_k(p)$, 
we see that $\overline{op^*}$ is parallel to $\eta$. 
Hence    by the uniform convexity of $\partial U^*_k$, 
\begin{equation*}
\beta = \arctan\left(\left|(\rho^{*}_k)'(p^*_2)\right|\right) - \arctan\big({\Small\text{$ \frac{p^*_1}{|p^*_2|} $}} \big) >0.
\end{equation*}

\begin{figure}[h]
\centering
\includegraphics[scale=0.33]{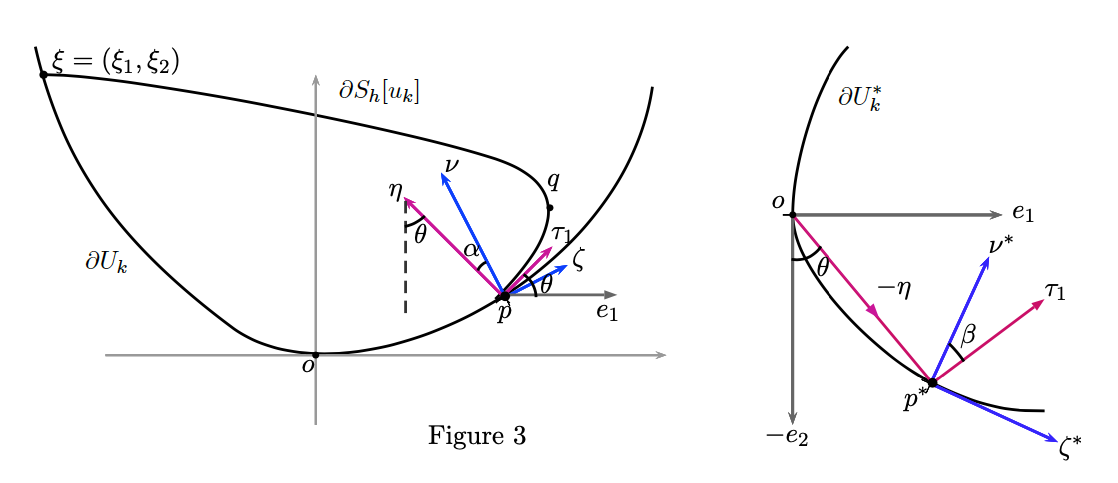}

{\Small ($\zeta, \zeta^*$ are tangential vectors, and $\nu$, $\nu^*$ are inner unit normals,  
of $\partial U_k$, $\partial U^*_k$ at $p$, $p^*=Du_k(p)$, respectively.  \\ 
$\eta$ and $\tau_1$ are the normal and tangential vectors  of $S_h[u_k]$ at $p$.  
Hence we have $\cos\alpha=\lan \eta, \nu\ran =\lan \tau_1, \zeta\ran$.)}
\end{figure}

Hence to prove Lemma \ref{1p}, we just need to prove that there is an interval of $\Gamma_k$ in which $\alpha\ge 0$.
From Figure 3, it is easy to see that $\alpha\ge 0$ if and only if $\p_\zeta u_k(p)\ge 0$.   
Since $u_k(0)=0$ and $u_k(x)>0$ $\forall\,x\in \Gamma_k\setminus\{0\}$,
there exists a point $\hat p=(\hat t, \rho_k(\hat t))\in\Gamma_k$ with $\hat t>0$ such that $\p_\zeta u_k(\hat p)> 0$.
Hence there exists $\eps>0$ such that 
\beq\label{n421}
\p_\zeta u_k(p)>0\ \ \text{for}\ p=(t, \rho_k(t))\in \Gamma_k, \ t\in (\hat t-\eps, \hat t+\eps).
\eeq
That means $\alpha\ge 0$ and so the obliqueness holds for $p=(t, \rho_k(t))\in \Gamma_k$ with  $t\in (\hat t-\eps, \hat t+\eps)$.

Let 
\beq\label{n422}
\bar t=\inf\left\{t :\ \text{the obliqueness holds at $p=(s, \rho_k(s))\in \Gamma_k$ for all  $s\in (t, t+\eps)$}\right\}.
\eeq
By the   $C^{1,\delta}$ regularity \cite {C92}, $\bar t\ge 0$ is well defined and 
the obliqueness holds for $t\in (\bar t, \bar t+\eps)$ but not at $\bar p=(\bar t, \rho_k(\bar t))$.
This finishes the proof. 
\end{proof}

\subsection{Contradiction}\label{ns44}
Now we derive a contradiction with the help of the approximation sequence $u_k$.
By our construction, $0\in\Gamma_k$ and $0\in \Gamma^*_k$,
where  $\Gamma_k$ and $\Gamma^*_k$ are the curves given in \eqref{Gak} and \eqref{Ga*k}.
First we point out that
\beq \label{1side}
\partial_{x_2}u_k(p)<0\ \ \ \text{for}\ \ p\in \Gamma_k\cap\{x_1>0\}.
\eeq
Indeed, for any given point $p\in\Gamma_k\cap\{x_1>0\}$, 
the inner normal $\nu$ of $U_k$ at $p$ lies in the second quadrant.
Let $\nu^*$ be the inner normal  of $U^*_k$ at $p^*=Du_k(p)\in \Gamma^*_k$.
By \eqref{n408} we have $\nu\cdot\nu^*> 0$.
By \eqref{rhok1}, it implies $p^*\in \Gamma^*_k\cap \{y_2< 0\}$.
Hence we have \eqref{1side}.

\begin{lemma}\label{u12}
We have 
	\begin{equation}\label{key3}
		\partial_{x_1x_2}u_k(t, \rho_k(t))<0 \ \ \ \text{ for }\ t\in (0, r_0).
	\end{equation}
\end{lemma}

\begin{proof}
By the boundary condition $Du_k(\partial U_k)=\p U_k^*$,
we have 
$$\partial_{_{x_1}}u_k(t, \rho_k(t))= \rho_k^*(\partial_{_{x_2}}u_k(t, \rho_k(t)) ). $$
Differentiating the above equation we have
$$
\partial_{_{x_1x_1}}u_{k} + \partial_{_{x_1x_2}}u_k\,\rho_k'=(\rho_k^*)' (\partial_{_{x_1x_2}}u_k+\partial_{_{x_2x_2}}u_{k}\,\rho_k'). 
$$
Namely,
\beq \label{aeq}
\partial_{_{x_1x_1}}u_{k} - \partial_{_{x_2x_2}}u_{k}\,\rho_k' \, (\rho_k^*)'=((\rho_k^*)'-\rho_k')\partial_{_{x_1x_2}}u_{k}.
\eeq
By the above approximation, $D^2u_k$ is  positive definite and
continuous on the boundary.
For $t>0$ small, by \eqref{1side} we have $\rho_k'> 0$ and $(\rho_k^*)'<0$. 
Hence, the LHS of \eqref{aeq} is always positive, 
and the coefficient on the RHS, $(\rho_k^*)'-\rho_k'<0$.
Therefore, we obtain \eqref{key3}.
\end{proof}

Introduce the function
	\begin{equation}\label{aux}
		w_k(x) := \partial_{_{x_1}}u_k+ u_k - x_1\partial_{_{x_1}}u_k.
	\end{equation}
By Lemma \ref{2dapp}, $\det\, D^2u_k$ is a positive constant.
Hence $w_k$ satisfies the equation
	\begin{equation}\label{aux1}
		M^{ij}D_{ij} w_k = 0
	\end{equation}
in $B_{r_0}\cap {U_k}$, where $\{M^{ij}\}$ is the cofactor matrix of $D^2u_k$.

\begin{corollary}\label{auco}
There exists a constant $\epsilon_0>0$ independent of $k$ such that for $t\in(0,\epsilon_0)$,
the function $w_k(t, \cdot)$ has an interior local minimum.
\end{corollary}

\begin{proof}
From Lemma \ref{u12} we have $\partial_{_{x_1x_2}}u_{k}(t, \rho_k(t)) < 0$ for $t\in (0, r_0).$ 
Hence, by \eqref{1side}
	\begin{equation}\label{infin}
		\partial_{_{x_2}}w_k(t, \rho_k(t)) = (1-t) \partial_{_{x_1x_2} }u_{k}(t, \rho_k(t)) + \partial_{_{x_2}}u_k(t, \rho_k(t)) < 0
	\end{equation}
for all $t\in (0, r_0)$.

On the other hand, by our assumption $U_k^*\subset\{y_1>0\}$, $\partial_{_{x_1}}u_k\geq0$. 
Hence by the strict convexity of $u_0$ and $u_k\rightarrow u_0$ uniformly in $B_{r_0},$ 
there exists a constant $\delta_0>0$ such that  
$$w_k(x)=u_k(x)+(1-x_1)\partial_{_{x_1}}u_k>u_k(x) \ge \delta_0$$
when $x\in \p B_{r_0}\cap U_k$.
Hence there is a small $\epsilon_0>0$ such that for any $t\in(0,\epsilon_0)$,
$w_k(t, \cdot)$ has a local minimum that is smaller than the boundary value $w_k(t, \rho_k(t))$.
\end{proof}

Hence we can define the following function
	\begin{equation}\label{aux2}
		\underline{w_k}(t) = \inf\{w_k(t, x_2) :\  x_2 > \rho_k(t),\  (t, x_2)\in U_k\}, \ \ t\in (0,\epsilon_0). 
	\end{equation}
By Corollary \ref{auco}, the infimum cannot be attained on $\p U_k\cap\{0<x_1<\epsilon_0\}$, and $\underline{w_k}$ is well defined for all  $t\in (0, \epsilon_0)$.

\begin{lemma}\label{leco}
$\underline{w_k}$ is concave for $t\in (0, \epsilon_0)$.
\end{lemma}

\begin{proof}
If $\underline{w_k}$ is not concave, 
then there exist constants $0<r_1<r_2<\epsilon_0$  and
 an affine function $L(t)$ such that 
$\underline{w_k}(r_i)= L(r_i)$ for $i=1,2$,  and the set 
  $\{t\in (r_1, r_2) : w_k(t)< L(t)\}\ne\emptyset$.
Extend $L$ to an affine function $\hat L$ defined in $\R^2,$ such that $\hat L(t, s)=L(t).$
Denote $$D_\eps =\{x\in U_k : x_1\in(r_1, r_2),\text{ and } w_k(x)< \hat L(x)-\eps\}.$$
By our definition of $\underline {w_k}$ and Corollary \ref{auco}, 
we can choose $\eps>0$ such that 
\beq\label{Deps}
\emptyset\ne D_\eps \Subset U_k.
\eeq
Indeed, by our choice of $\hat L$, $D_{\eps\,|\,\eps=0}\ne\emptyset$.
Let $\eps_1=\sup\{\eps : D_{\eps}\ne\emptyset\}$.
Then if $\eps<\eps_1$ and sufficiently close to $\eps_1$, 
we have $D_\eps \ne \emptyset$. 
By Corollary \ref{auco}, the infimum in \eqref{aux2} is attained at an interior point.
Hence we also have $D_\eps \Subset U_k$.

By equation \eqref{aux1} and the boundary condition $w_k=\hat L-\eps$ on $\partial D_\eps$,
we apply the maximum principle to $w_k$ in $D_\eps$ and conclude that  $w_k= \hat L-\eps$ in $D_\eps$.
However,  by our definition of $D_\eps$, we have $w_k<\hat L-\eps$ in $D_\eps$.
We reach a contradiction. 
\end{proof}

Note that for any fixed $t \in (0, \epsilon_0)$,  the minimum point in Corollary \ref{auco} may not  be unique. 
In this case, the domain $D_\eps$ in the above proof may contain more than one connected component.
But each component is compactly contained in $U_k$. Hence we can still use the maximum principle.

We have now established Lemmas \ref{u12} and \ref{leco} for the
approximation sequence $u_k.$  Denote
$$w_0= \p_{x_1}u_0 + u_0 - x_1\p_{x_1}u_0=\lim_{k\to\infty} w_k . $$
Let
$$\underline{w_0}(t) = \inf\{w_0(t, x_2) :\  x_2 > \rho_0(t),\  (t, x_2)\in \Om_0\}, \ \ t\in (0,\epsilon_0). $$   
From \eqref{key5}, we have
$\underline{w_0}(t) \to 0$ as $ t\to 0$. More precisely, 
$$\underline{w_0}(t) \leq C\sigma(2t)\ \ \text{ for $t>0$ small}. $$
To see this, let $q=(2t, q_2)$ such that $\underline{u_0}(2t)=u_0(q)$,
where $\underline{u_0}$ is defined in Corollary \ref{u0est}.
By \eqref{key5}, we have $u_0(q)\leq Ct\sigma(2t)$. 
Let $\hat q=(t, q_2)$. 
Since $\partial_{x_1}u_0\geq 0$ (i.e. $\Om^*_0\subset\{y_1>0\}$), we have
$$u_0(\hat q)\leq u_0(q)\leq Ct\sigma(2t). $$
By the convexity of $u_0$, one also has
$$\partial_{x_1}u_0(\hat q)\leq \frac{u_0(q)-u_0(\hat q)}{t} \leq C\sigma(2t) . $$ 
Hence,
	\begin{equation}\label{2dbb}
	\begin{split}
    \underline{w_0}(t) 
		& = \inf w_0(t,\cdot)  \leq w_0(\hat q) \\
		&= (1-t)\partial_{x_1}u_0(\hat q) + u_0(\hat q) \leq C\sigma(2t).
	\end{split}
	\end{equation}
Recall that $\sigma(t)=|t|^{1+\gamma}$. 
Hence by the concavity of $\underline{w_0}$ (Lemma \ref{leco}, taking the limit $k\rightarrow \infty$),
we conclude that $\underline{w_0}(t)\leq0$ for all $t\in(0, \epsilon_0)$. 
On the other hand, since $\partial_{x_1}u_0\geq 0$, by the strictly convexity of $u_0$, we have
$$ w_0(x)   =u_0(x) +(1-x_1)\p_{x_1}u_0(x) \ge u_0 (x)>0 \ \ \text{if}\ x\in\Om_0, \ x_1\in (0, \epsilon_0) .$$
It implies that $\underline{w_0}(t) >0 $ when $t>0$. 
We reach a contradiction. 
Therefore the uniform obliqueness in dimension two is proved.
\qed

\section{Uniform obliqueness in high dimensions}\label{ss42}
 
In this section we prove the uniform obliqueness in high dimensions. 
Suppose the domains $\Omega$ and $\Omega^*$ are bounded, convex, with $C^{1,1}$ boundaries, and $f\in C^{0}(\overline{\Omega}).$
The proof of obliqueness uses the ideas in \S\ref{s4}, and also the following:

\begin{itemize}
\item [(i)] 
Suppose the obliqueness fails at the point $0\in \partial \Omega$ and $0=Du(0)\in\partial\Omega^*$.
To understand the geometry of the sub-level set $S_h[u](0)$,
we prove that $S_h[u](0)$ is contained in a cuboid $Q$ with volume $|Q|\le C|S_h[u](0)|$.
Then by a rescaling of the coordinates such that $Q$ becomes a cube, 
$S_h[u](0)$ changes accordingly to a convex set with good shape. 
In particular, we show that the boundary $\partial\Omega$ becomes flat in directions orthogonal to the inner normals $\nu$ and $\nu^*$  as $h\rightarrow 0$.
This property enables us to employ the techniques in \S4.  
 
\item [(ii)] As in dimension two, we need to construct a smooth approximation sequence to derive the contradiction. 
The construction in high dimensions is more complicated. 
\end{itemize}

Similarly as in Remark \ref{R4.1} (iii), 
the constants in this section depends on $n, \lambda, \Om, \Om^*$ 
(diameters  and $C^{1,1}$ norm of $\Om, \Om^*$). 
They also depend on $k$ in the argument on the approximation sequence $\{u_k\}$,
but are independent of $h$ and $u$ for small $h>0$.
The continuity of $f$ is used only in the blow-up process.
The tangential $C^{1,\alpha}$ regularity of \S3 is used only in \S5.1 for $u$ and in the approximation sequence $\{u_k\}$,
where the domains are $C^{1,1}$ smooth. 
We do not need the tangential $C^{1,\alpha}$ regularity for the limit $u_0$.

\subsection{The limit profile} \label{sn51}

To prove the uniform obliqueness, by the $C^{1,\delta}$ regularity \cite {C92},
we may suppose to the contrary that $0\in\pom$, $u(0)=0$, $Du(0)=0\in\p\Om^*$, and
locally
\beq\label{hdom} 
{\begin{aligned}
	\{x_n > C|x'|^2\} \subset \Omega \subset \{x_n > 0\}, & \quad\mbox{ where } x'=(x_1,\cdots,x_{n-1}), \\
	\{ y_1> C|\tilde y|^2 \} \subset \Omega^* \subset \{y_1>0\}, & \quad\mbox{ where } \tilde y=(y_2,\cdots,y_n).
	\end{aligned}}
\eeq 

Corresponding to properties $(i)$ and $(ii)$ in \S\ref{ss41},  similarly we have
\begin{enumerate}[$(i)$]
\item $u_1>0$ in $\Omega$ and $v_n>0$ in $\Om^*$;

\item if $x\in S_h[u]$, then $x-te_1\in S_h[u]\ $ $\forall\ t>0$, provided $x-te_1\in\Omega$.
\end{enumerate}

First we prove a lemma that strengthens Lemma \ref{lemban}.

\begin{lemma}\label{highl}
For any given point $p\in \p S_h[u]\cap \Om$, 
let $\mathcal H$ be the tangential plane of $S_h[u]$ at $p$.
Assume $\mathcal H=\{x\in\R^n:\ x\cdot \gamma= a\}$ for a unit vector $\gamma$,
where $a$ is a positive constant.
Then 
\beq\label{4f6}
x\cdot \gamma\ge -Ca \quad \ \forall\ x\in S_h[u],
\eeq
where $C>0$ is a constant independent of $h$ ($h>0$ small) and $u$.
\end{lemma}

\begin{proof} Denote $b=\inf\{x\cdot \gamma:\  x\in S_h[u]\}$ and denote
$\mathcal H_1=\{x\in\R^n:\ x\cdot \gamma= b\}$.
Suppose to the contrary that the ratio $\frac{a}{|b|}\to0$ as $h\to0$.
Let $\mathcal{A}_h$ be an affine transformation  
such that $\mathcal{A}_h(S_{h}[u])\sim B_1(z)$ for some point $z$.
Note that the transform does not change the ratio $\frac{a}{|b|}$. 

Accordingly let $u_h(x) = u(\mathcal{A}_h^{-1}x)/h$. By Caffarelli's geometric decay estimate, 
similarly as in \eqref{ulc}, $u_h$ is locally uniformly bounded and sub-converges to a limit $u_0$ as $h\to0$, 
and $u_0(0)=\lim_{h\to0} u_h(0)=0$. 

On the other hand, by passing to a subsequence, we have
	\begin{equation*}
		\mathcal{A}_h\left(\mathcal H\right) \to \mathcal H^*,\qquad
		 \mathcal{A}_h\left(\mathcal H_1\right) \to \mathcal H_1^*
	\end{equation*}
as $h\to 0$. Observe that
	\begin{equation*}
		\frac{\dist(0, \mathcal H^*)}{\dist(0, \mathcal H^*_1)} = \lim_{h\to0}\frac{a}{|b|}=0.
	\end{equation*}
Hence $\mathcal H^*$ passes through the origin. 
It implies that $u_0\ge 1$ on one side of $\mathcal H^*$, which is a contradiction since $u_0$ is continuous and $u_0(0)=0$.  
\end{proof}

Similarly as in \eqref{qxi}, let $q, \xi \in \p S_h(0)$ such that 
	\begin{align}\label{dqq}
		& q_1=\langle q, e_1 \rangle = \sup\{\langle x, e_1\rangle : x\in S_h(0)\}, \\
		& \xi_1 = \langle \xi, e_1 \rangle = \inf\{\langle x, e_1\rangle : x\in S_h(0)\}.
	\end{align}
Obviously $q_1>0$ and $\xi_1 <0$. 
We point out that $q$ is an interior point of $\Om$.
To see this, let $\ell^*=:\{te_1:\ t\in (0,t_0)\}$ be a line segment in  $\Om^*$. 
Let $\ell=(Du)^{-1}(\ell^*)\subset \Om$ be the pre-image of $\ell^*$.
Then $q$ is the intersection of $\ell$ with $\p S_h[u](0)$.
From Lemma \ref{highl} we have

\begin{corollary}\label{rl42} 
We have
	\begin{equation}\label{banban}
		q_1 \geq \delta_0|\xi_1|
	\end{equation}
for some constant $\delta_0>0$ independent of $h$ and $u$. 
\end{corollary}

{Having obtained the balance property \eqref{banban}, 
one would expect a decay estimate like Corollary \ref{co41} as in dimension two. 
But to obtain the decay estimate in high dimensions, one needs a bound of $S_h[u]$ along $x_2,\cdots,x_{n-1}$ directions.}

Denote $S^c_{h;1}[u]=S^c_h[u]\cap\{x_1=0\}$
and denote $B'_r(0)$ the ball of radius $r$ in $\R^{n-2}=\text{span}(e_2, \cdots, e_{n-1})$.
We have the following estimates for $S^c_h[u]$ in $x_2,\cdots,x_{n-1}$ directions.

\begin{lemma}\label{lem4.6}
For any given $\eps>0$ small, we have 
\beq\label{n003}
B'_{C^{-1} h^{1/2+\eps}}(0)\subset S^c_{h;1}[u]\cap \{x_n=0\}\subset B'_{ C h^{1/2-\eps}}(0).
\eeq
In particular, for any $i=2,\cdots,n-1$ and any $x\in S^c_h[u]$,
\beq\label{n100}
|x_i|=|x\cdot e_i| \le Ch^{\frac 12-\eps}, 
\eeq
provided $h>0$ is sufficiently small, where the constant $C=C(\eps)$ is independent of $h$ and $u$.
\end{lemma}

\begin{proof}
By the tangential $C^{1,1-\eps}$ regularity for $u$, we have 
$S^c_{h;1}[u]\cap\{x_n=0\}\supset B'_{C^{-1}h^{1/2+\eps}}(0)$, which is the first inclusion of \eqref{n003}.

For any points $x\in S^c_h[u]$ and $y\in S^c_h[v]$,
by \eqref{inn} we have $|x\cdot y|\le Ch$.
By the tangential $C^{1,1-\eps}$ regularity for $v$, 
we have $C^{-1}h^{\frac 12+\eps}e_i\in S^c_h[v], (i=2, \cdots, n-1)$. 
Hence $(h^{\frac 12+\eps}e_i)\cdot x\le Ch$ $\forall \ x\in S^c_h[u]$, namely
\beq\label{n012}
|x_i|=|x\cdot e_i| \le Ch^{\frac 12-\eps}\ \ \forall\, x\in S^c_h[u].
\eeq
We obtain \eqref{n100} and the  second inclusion of \eqref{n003}.
\end{proof}

From \eqref{n100} and Remark \ref{cwidth}, we also obtain 
\beq\label{n666}
	|x\cdot e_i| \le Ch^{\frac 12-\eps}\ \ \forall\, x\in S_h[u]  \text{ and } i=2,\cdots,n-1.
\eeq
The same estimate is true for $S_h[v], S^c_h[v]$ as well.  
From \eqref{n666} we will derive two consequences: one is a decomposition \eqref{cooo}, 
the other one is the decay estimate in Corollary \ref{keyco}.

We shall first derive the decomposition \eqref{cooo}.
Note that in high dimensions, there may be a small portion of $S_h[u]\cap\{x_1>0\}$, whose projection on the plane $\{x_1=0\}$ 
is not contained in $S_{h;1}:=S_h[u]\cap \{x_1=0\}$. 
Nevertheless, we have the following inclusion.

\begin{corollary}\label{copp}
Let $S'_{h;1}[u]$
be the projection of $S_h[u]\cap\{x_1>0\}$ on the plane $\{x_1=0\}$.
Then  
	\begin{equation}\label{12good}
		S'_{h;1}[u]\subset (1+o(1))S_{h;1}[u]
	\end{equation}
as $h\to0$, where the dilation is with respect to $z$, the centre of $S_{h;1}[u].$ 
\end{corollary}

\begin{proof}
Let $\tilde{x}=(0, x'', x_n)\in  S'_{h;1},$ where $ x''=(x_2,\cdots, x_{n-1})$.
By definition of $S'_{h;1}[u]$, there is $t>0$ such that $x=(t, x'', x_n)\in S_h[u]$ and $u(x) < h$. 
{If $\tilde x\in\Omega$, since $u_1>0$, one must have $u(\tilde x) < h$, and thus $\tilde x\in S_{h;1}[u]$.}
 
If $\tilde x\notin\Omega$, let $z$ be the centre of $S_{h;1}[u]$. 
By the $C^{1,\delta}$ regularity of $u$, 
we have $z_n=z\cdot e_n  \ge C h^{\frac{1}{1+\delta}}$.
From \eqref{n666}, 
$|x''|\le  Ch^{\frac{1}{2}-\epsilon}$.
Since $\tilde x\notin\Omega$ and $\partial\Omega\in C^{1,1}$,
\beq\label{n404}
 x_n \le C|x''|^2 \le C h ^{1-2\epsilon} =o(h^{\frac{1}{1+\delta}})=o(z_n).
 \eeq
 Let $\ell$ be the segment connecting $z$ and $\tilde x$,
and let $y$ be the intersection of $\ell$ and $\partial\Omega.$ 
{Since $u(z)<h$, $u(\tilde x)<h$, we have $u(y)<h$ and thus $y\in S_{h;1}[u].$}
Write $y=(0,y'',y_n)$. 
 By \eqref{n666} again, we have $|y''|  \le C h^{\frac{1}{2}-\epsilon}.$ Then since $y\in \partial\Omega$, one has $y_n \le C h ^{1-2\epsilon} \ll z_n.$
Therefore, 
\begin{equation*}
\lim_{h\to 0} \frac{|z\tilde x|}{|z y|} = \lim_{h\to 0} \frac{|z_n-x_n|}{|z_n-y_n|} =1,
\end{equation*}
from which one easily obtains \eqref{12good}. 
\end{proof}

Next we estimate the size of $S_h[u]\cap \{x_1<0\}$.
We introduce a cone with vertex $q$ (see \eqref{dqq}) and passing through $S_{h;1}[u]$, namely
 $$\mathcal{V}=\{q+t(x-q): x\in S_{h;1}[u], t\geq0\}.$$
By the convexity of $S_h[u]$, we have 
$S_h[u]\cap\{x_1<0\}\subset \mathcal{V}$.
Hence by Corollaries \ref{rl42} and \ref{copp},
\begin{equation}\label{cooo}
S_h[u] \subset [\xi_1, q_1] \times \beta S_{h;1}[u]
\end{equation}
for a constant $\beta >0$ independent of $h$,
where $\beta S_{h;1}[u]$ denotes the $\beta$-dilation with respect to the centre of $S_{h;1}[u]$. 
{Indeed, by performing an affine transform in the $e_2,\cdots, e_n$ directions, we may assume $S_{h;1}[u]$ is 
normalised. Then by  Corollary  \ref{copp}  we have that $q'=(0, q_2,\cdots, q_n)\in (1+o(1))S_{h;1}[u].$ Hence, by 
Corollary \ref{rl42} and using the fact that $S_h[u]\cap\{x_1<0\}\subset \mathcal{V}$ we have \eqref{cooo}.

 \begin{remark}\label{applytov}
 Replacing the $e_1$-direction by the $e_n$-direction,
the same argument for \eqref{cooo} also applies to $S_h[v]$ and yields 
 \begin{equation}\label{cooo1}
S_h[v] \subset [\xi^*_n, q^*_n] \times \beta^* S_{h;n}[v]
\end{equation}
for a constant $\beta^*>0$ independent of $h$,
where $\xi^*, q^* \in \partial S_h[v]$ is defined analogously to \eqref{dqq} (where $e_1$ is replaced by $e_n$) and
$S_{h;n}[v]:=S_h[v]\cap\{y_n=0\}$. 
 \end{remark}

As another consequence of \eqref{n666}, we next derive a decay estimate analogous to Corollary \ref{co41}.  
 
\begin{lemma}\label{co61}
 For any given $\varepsilon>0$ small, we have $q_1\geq h^{\frac13+\varepsilon}$, provided $h$ is sufficiently small.
\end{lemma}
 
\begin{proof}
For any $x\in S_h[u],$ by \eqref{n666} we have $|x_i|\leq Ch^{\frac{1}{2}-\epsilon}$ for $i=2,\cdots, n-1$. 
By Corollary \ref{rl42} we also have $q_1\geq C|x_1|.$
Since $u_1>0$, we see that 
$ \sup\{e_n\cdot x:\ x\in S_h[u]\}$ must be attained on the boundary $\partial\Omega$.
Since $\partial\Omega\in C^{1,1}$,  we have
	\begin{equation*}
		x_n \leq C\sum_{i=1}^{n-1}x_i^2 \leq C(q_1^2 + h^{1-2\varepsilon}) \quad \ \ \forall\ x\in S_h[u]\cap\partial\Omega.
	\end{equation*} 
From \eqref{gV2},  the volume $|S_h[u]|\approx h^{\frac{n}{2}}$.
Hence 
	\begin{equation*}
		h^{\frac{n}{2}} \approx |S_h[u]|\leq Cq_1(q_1^2+h^{1-2\varepsilon})h^{\frac{n-2}{2}-(n-2)\varepsilon}.
	\end{equation*}
Therefore $q_1\geq h^{\frac13+\varepsilon}$  for any given $\varepsilon>0$ small. 	
\end{proof}

From Lemma \ref{co61}, 
similarly to \eqref{key1},
we have the following corollary.
\begin{corollary}\label{keyco}
For $t>0$ small, denote
		\begin{equation*}\begin{split}
			\underline{u}(t) &= \inf\{u(t,x_2,\cdots,x_n) :  (t,x_2,\cdots,x_n)\in\Omega\}, \\
			\underline{\partial_1u}(t) &= \inf\{\partial_1u(t,x_2,\cdots,x_n) :  (t,x_2,\cdots,x_n)\in\Omega\}. 
		\end{split}
		\end{equation*}
	We have the asymptotic behaviour 
	\begin{equation}\label{keyh} 
	\begin{split}
		\underline{u}(t) \leq C t^{3-\varepsilon}, \\
		\underline{\partial_1u}(t) \leq C t^{2-\varepsilon}    
	\end{split}
	\end{equation}
for $t>0$ small, where $\varepsilon>0$ is any given small constant. 
\end{corollary}

\begin{remark}\label{upest}
By Lemma \ref{rela1} and Lemma \ref{co61} we have 
	$$s\gtrsim q_1\gtrsim h^{\frac13+\varepsilon},\quad\text{ where } s:=\sup\{x\cdot e_1 :  x\in S_{bh}^c[u]\}.$$ 
Let $d=\sup\{x\cdot e_n: x\in S_{bh}^c[u],\, x_1=0\}.$ Then by Lemma \ref{lem4.6} and \eqref{gV1}, we have 
$$h^{\frac{n}{2}}\approx |S_{bh}^c[u]|\gtrsim h^{\frac{n-2}{2}(1+\epsilon)}sd,$$
which implies $d\lesssim h^{\frac{2}{3}-\epsilon}.$ By Lemma \ref{rela1} again, we obtain that 
\beq\label{rembd}
	\sup\{x\cdot e_n: x\in S_{h;1}[u]\}\lesssim h^{\frac{2}{3}-\epsilon}.
\eeq
\end{remark}

In order to bound the sub-level set $S_h[u]$ by a cuboid, we need to further decompose $S_{h;1}[u]$ in \eqref{cooo} along $e_n$ direction.  
Denote 
	$$S^c_{h; 1, 0}=S^c_{h;1}[u]\cap\{x_n=0\},$$
where $S^c_{h;1}[u]=S^c_h[u]\cap\{x_1=0\}$ was introduced above.

\begin{lemma}\label{l405} 
Let $P_h$ be the projection of $S_{h;1}[u]$ on $\{x_n=0\}.$ Then we have
\beq\label{4f7}
P_h \subset \beta S^c_{h; 1, 0}
\eeq
for a constant $\beta$ independent of $h$ and $u$.
\end{lemma}
\begin{proof} 
Let $e \in\text{span}\{e_2,\cdots, e_{n-1}\}$ be a unit vector, and denote $r_e:=\sup\{t : te\in S^c_{h;1,0}[u]\}.$
To prove \eqref{4f7}, it suffices to show that 
\begin{equation}\label{ne66}
 |x\cdot e|\le \beta r_e \quad \forall\, x\in S_h[u]\cap \text{span}\{e_n, e\}
 \end{equation}
 for all unit vectors $e \in\text{span}\{e_2,\cdots, e_{n-1}\}$.
 By Lemma \ref{lem4.6}, we have $r_e\ge C^{-1} h^{\frac{1}{2}+\epsilon}.$

Given a unit vector $e \in\text{span}\{e_2,\cdots, e_{n-1}\},$ and a point $p\in S_h[u]\cap \text{span}\{e_n, e\},$
up to a rotation of coordinates,  we assume $e=e_2$ and $p=(0,p_2,0,\cdots,0, p_n)$ with $p_2>0$. 
By Remark \ref{upest},
we have 
\begin{equation}\label{xc2}
 p_n\le C h^{\frac{2}{3}-\epsilon}
\end{equation}
 for $\epsilon$ as small as we want.
 In order to prove \eqref{ne66}, it suffices to show that $p_2 \le \beta r_{e_2}$.
 If $p_2\ll h^{\frac{1}{2}+\epsilon},$ we readily have $p_2=p\cdot e_2\leq r_{e_2}$. 
 Hence it suffices to consider the case
 	\beq\label{keypt} 
		p_2\ge C h^{\frac{1}{2}+\epsilon} \gg p_n\quad\text{ (and thus $\ p_2\approx |p|\ $ for $h$ small)}. 
	\eeq 

By Remark \ref{dforsh}, we have 
 \begin{equation}\label{simp2}
|y\cdot p|\leq Ch \quad \forall\, y\in S_h[v].
\end{equation}
In particular, when $y\in S_{h;n}[v]:=S_{h}[v]\cap \{y_n=0\},$ $y\cdot p = y_2p_2$. Thus we obtain 
	 \begin{equation*}
	\sup\{|y_2|:y\in S_{h;n}[v]\} \le C\frac{h}{p_2} .
	\end{equation*}
 By Remark \ref{applytov}, $\sup\{|y_2|:y\in S_{h}[v]\}\le \beta^* \sup\{|y_2|:y\in S_{h;n}[v]\}$. Therefore, we obtain
	\begin{equation} \label{y2est}
		\sup\{|y_2|:y\in S_{h}[v]\} \le C\beta^* \frac{h}{p_2}.
	\end{equation}

By the definition of $r_{e_2},$ we have $r_{e_2}e_2\in \partial S^c_h[u]$. 
Hence by \eqref{inn1}, there exists $z^*\in \partial S^c_{h}[v]$ such that 
	\begin{equation}\label{cn111}
		z^*\cdot (r_{e_2}e_2) \ge C^{-1} h.
	\end{equation}
By \eqref{y2est} and Remark \ref{cwidth}, we have
\begin{equation}\label{cn222}
z^*\cdot e_2\leq \sup\{y\cdot e_2:y\in S^c_h[v]\} \le C\beta^* \frac{h}{p_2}.
\end{equation}
Hence from \eqref{cn111} and \eqref{cn222}, we obtain the desired inequality 
	\begin{equation}\label{nn6}
		r_{e_2} \ge
		 \frac{C^{-1}h}{z^*\cdot e_2} \ge \frac{1}{C\beta^*} p_2,
	\end{equation}
for a different constant $C>0$. 
This finishes the proof with $\beta=C\beta^*$. 
\end{proof}

Thanks to \eqref{cooo} and Lemma \ref{l405}, we can now show that $S_h[u]$ is contained in a cuboid as follows.  
Denote 
\begin{equation}
\label{defq1}
 d_{n}= \sup\{e_n\cdot x:\ x\in S^c_{h; 1}[u]\}
 \end{equation}
to be the height of  $S^c_h[u]$  on the section $\{x_1=0\}$.
We have 
\begin{align}\label{n101} 
d_n&\gtrsim \sup\{e_n\cdot x:\ x\in S_{b^{-1}h; 1}[u]\} \nonumber \\
&\gtrsim \sup\{e_n\cdot x:\ x\in S_{b^{-1}h}[u]\}\\ 
&\gtrsim  \sup\{e_n\cdot x:\ x\in S_{h}[u]\}\nonumber, 
\end{align}
where the first inequality is due to Lemma \ref{rela1}, the second inequalities follows from \eqref{cooo},
and  the last inequality is due to the convexity of $u,$ which implies that $S_h[u]\subset bS_{b^{-1}h}[u]$, $(b>1)$.

Let $\tilde q\in \p S^c_h[u]$ be the point such that 
 \begin{equation}\label{defq2}
 \tilde q_1= \tilde q\cdot e_1 = \sup\{ x\cdot e_1: x\in S^c_h[u]\}.
 \end{equation}
 By Remark \ref{cwidth} we have that  
 \begin{equation}\label{defq3}
 \sup\{|x\cdot e_1| : x\in S_h[u]\}\lesssim \tilde{q}_1.
 \end{equation}

Let  
\beq\label{cylin0}
\mathcal R_h=[-\tilde q_1, \tilde q_1]\times E'_h\times [-d_n,  d_n]
\eeq
be a cuboid, 
where $E'_h\subset \R^{n-2}$ is an ellipsoid centred at $0$ such that  $E'_h\sim S^c_{h; 1, 0}=S^c_{h; 1}[u]\cap\{x_n=0\}$. 
By Lemma \ref{l405}, \eqref{cooo} and \eqref{defq3}, we have  
\begin{equation}\label{vol11}
S_{h}[u]\subset C \mathcal R_h 
\end{equation} 
for some constant $C$ independent of $h.$
Moreover, by  \eqref{defq1} and \eqref{defq2}  the volume 
$$ |S^c_h[u]|\gtrsim |E'_h|d_n\tilde{q}_1\gtrsim |\mathcal R_h|.$$ 
Hence, 
\beq\label{n0010}
C^{-1} |\mathcal R_h|\le |S^c_h[u]|\approx |S_{h}[u]|\le  C|\mathcal R_h|. 
\eeq

Now we make a linear transform 
$\mathcal T=\mathcal T_2\circ \mathcal T_1$ 
such that the sub-level set $S_h[u]$ has a good shape, 
where $\mathcal T_1$  is a linear transform normalising $E_h'$ to the unit ball in $\R^{n-2}=\text{span}(e_2, \cdots, e_{n-1})$ 
while leaving $x_1$ and $x_n$ unchanged; and $\mathcal T_2$ is given by
\beq \label{n001}
\mathcal T_2:\ \ \ 
\Big\{ {\begin{aligned}
 & \tilde x_1=x_1/\tilde q_1, \hskip10pt \tilde x_n=x_n/d_n ,\\[-3pt]
 & \tilde x_i=x_i\ \ \ \  \text{for}\ 2\le i\le n-1. 
 \end{aligned}}
 \eeq
It is easy to see that $\mathcal T_2\circ \mathcal T_1=\mathcal T_1\circ \mathcal T_2$.

After the transform $\mathcal T$, the set $\mathcal{T}(S_h[u])$ is contained in the cube 
$\mathcal D=[-C, C]^n$, and the volume $|\mathcal T(S_h[u])|\ge \delta_0$ 
for a positive constant $\delta_0$ independent of $h$ and $u$.
Hence $\mathcal T(S_h[u])$ has a good shape. By Lemma \ref{rela1} and Lemma \ref{luni}, we see that 
$\mathcal T(S^c_h[u])$ also has a good shape.
By rescaling back and using Lemma \ref{rela1} again, we have
\beq\label{mmeq}
	C^{-1}\mathcal{R}_n \cap \Om \subset S_h[u] \subset C\mathcal{R}_n
\eeq
for a constant $C$ independent of $h$. 

Having made the transform $\mathcal T$ (note that $\mathcal T=\mathcal T_h$ depends on $h$), accordingly 
we also make the change $u_h(x) = u(\mathcal T^{-1}x)/h$.

Let $\underline{u}(t)$ be the function introduced in Corollary \ref{keyco}.
Similarly to \eqref{subs}, we choose a sequence $\{t_j\}\to0$ such that
	\begin{equation}\label{newcho}
		\frac{\underline{u}(t)}{\underline{u}(t_j)} \leq 2\Big(\frac{t}{t_j}\Big)^{3-\varepsilon} \qquad\forall\, t\in(0,t_j)  ,
	\end{equation}
where $\varepsilon>0$ is the small constant in \eqref{keyh}.
Denote $\mathcal T_j=\mathcal T_{h_j}$, $u_j=u_{h_j}$, where $h_j=\underline{u}(t_j)$.

Similarly as in \S\ref{s4}, by passing to a subsequence, $\Omega_{h_j}:=\mathcal T_j(\Omega)$ converges to a limit $\Omega_0$ as $j\to\infty$, and $\Om_0$ is an unbounded convex domain in $\mathbb{R}^n$. 
Also, $u_j$ converges to a limit $u_0$ as $j\to\infty$, which satisfies the Monge-Amp\`ere equation 
\eqref{eq00} in $\mathbb{R}^n$.

By the proof of Corollary \ref{u0est},
$u_0$ satisfies the asymptotic behaviours \eqref{keyh}.
Moreover, $u_0$ is strictly convex and $C^{1,\alpha}$ regular in $B_k\cap \overline{\Omega_0}$  for any $k>0,$ and $u_0\in C^1(\mathbb{R}^n),$ (see Remark \ref{r401}).

Thanks to the above cuboid decomposition \eqref{cylin0}, we can prove that the boundary $\partial\Omega_0$ is flat in $x_2, \cdots, x_{n-1}$ directions.

\begin{lemma}
\label{god1}
Assume $\Om_h$ sub-converges as $h\to0$ to a convex domain $\Om_0$, locally in the sense of Hausdorff.
Then $\Om_0=\R^{n-2}\times \omega_0$, where $\omega_0$ is a convex set in the 2-dim space $\text{span}\{e_1, e_n\}$.
\end{lemma}

\begin{proof}
By the global $C^{1,\delta}$ regularity \cite {C92}, 
we have $d_n\ge C h^{\frac{1}{1+\delta}}$ for some $\delta>0$, where $d_n=d_{n}(h)$ is given in \eqref{defq1} and \eqref{n101}.
By \eqref{n003}, we have 
$$S^c_{h;1,0}[u]=S^c_{h; 1}[u]\cap\{x_n=0\}\subset B'_{Ch^{1/2-\eps}}(0).$$ 
Hence by the $C^{1,1}$ regularity of the boundary $\pom$,
the height of $S^c_{h;1}[u]\cap \pom$ satisfies
$$d_{n,b}(h)=: \sup\{e_n\cdot x:\ x\in S^c_{h; 1}[u]\cap\pom\}\le Ch^{1-2\eps},$$
where $\eps>0$ is fixed but can be as small as we want. 
Hence, 
$$\frac {d_{n,b}(h^{1- \delta/2})}{d_{n}(h)}\to 0\ \ \ \text{as}\  h\to 0. $$
Note that by \eqref{n003}, 
$$\text{diam}(S^c_{h; 1,0}[u])= o(1) \text{diam}(S^c_{h^{1-\delta/2 }; 1,0}[u])\quad\text{ when } h\to0. $$
The above formula implies that  
$\mathcal{T}_h(\pom\cap S^c_{h^{1-\delta/2 }; 1}[u])$ is becoming flat and so its limit is the 
plane $\text{span}(e_2, \cdots, e_{n-1})$.
Namely $\mathcal{T}_h(\pom\cap B_R(0)\cap\{x_1=0\})$ becomes flat as $h\to 0$, $\forall\, R>0$.

It is well known that if a convex set $G$ contains a straight line $\ell$, then $G$ can be expressed as a product 
$G=G'\times \ell$. The lemma is proved.
\end{proof}

Denote $\mathcal T^*=\frac{1}{h}(\mathcal T')^{-1}$, the dual affine transformation for $v$, {where $\mathcal T'$ is the transpose of $\mathcal T$}.
Similarly, we denote $v_h(y)=\frac 1h v((\mathcal T^*)^{-1}y)$, and $\Omega^*_h=\mathcal T^*(\Om^*)$.
Applying Lemma \ref{god1} to $v$, 
we see that $\Omega^*_h$ converges to an unbounded convex domain $\Omega_0^*=\R^{n-2}\times \omega^*_0$, 
where $\omega^*_0\subset span\{e_1, e_n\}$ is a convex set. 
And $v_{h_j}$ converges to a convex function $v_0$ locally uniformly, which satisfies the equation \eqref{eq000} correspondingly.

\subsection{Smooth approximation}\label{hdsm}

First we construct a smooth approximation sequence $\{u_k\}$ converging to $u_0$ in a small neighbourhood of the origin similarly as in \S\ref{ns43}.

Let $V=B_r\cap \Omega_0^*$ for a small constant $r,$  and let  $U=Dv_0(V).$ 
Then $B_{r_0}\cap\Omega_0\subset U$ for a small constant $r_0>0$. 
By Lemma \ref{god1} we approximate  $U, V$
by a sequence of bounded smooth  sets $U_k, U_k^*$  respectively such that
$$ {\begin{split}
  a)\ \ & 0\in \p U_k, \ \  0\in \p U^*_k \ \ \forall\ k\ge 1,  \\
  b)\ \  & U_k\subset \{x_n>0\}, \ \ U^*_k\subset \{y_1>0\} \ \ \forall\ k\ge 1, \\
  c)\ \  & U_k \to U, \ \   U^*_k \to V, \ \ \text{as $k\to\infty$, in Hausdorff's sense},\\
  d)\ \ & \exists \text{ smooth, uniformly convex sets } \hat\omega_k, \hat\omega_k^* \subset \text{span}\{e_1, e_n\} \text{ such that} \\
\ \ & \ \ \ \qquad U_k \cap B_{r_0} =(\hat\omega_k \times \mathbb{R}^{n-2})\cap  B_{r_0}, \\ 
\ \ & \ \ \ \qquad U^*_k  \cap B_{r_0} = (\hat\omega_k^*\times \mathbb{R}^{n-2})\cap  B_{r_0} \\
\ \ & \text{ for a different, smaller constant $r_0>0$. } 
  \end{split}} $$
In this subsection, the constant $r_0$ may change from line to line but they have a uniform positive lower bound independent of $k.$

Moreover, we assume the lower part of the boundary $\p U_k\cap B_{r_0}$ is the graph of a smooth, uniformly convex function $\rho_k$ in direction $e_n$, that is
\beq\label{hdrhok}
\Gamma_k =: \{x\in\R^n : x_n=\rho_k(x_1)\} \cap B_{r_0}, 
\eeq
where $\rho_k$ is defined on $\mathcal{J}_k$ that is the projection of $\p U_k\cap B_{r_0}$ on the $x_1$-axis with $[0,\frac12r_0)\subset\mathcal{J}_k$.
The function $\rho_k$ satisfies \eqref{rhok1} as in dimension two. 

Similarly, the left part of the boundary $\p U^*_k\cap B_{r_0}$ is the graph of a smooth, uniformly convex function $\rho^*_k$ in direction $e_1$, that is
\beq\label{hdrhok*}
\Gamma^*_k =: \{y\in\R^n : y_1=\rho_k^*(y_n)\} \cap B_{r_0}, 
\eeq
where $\rho^*_k$ is defined on $\mathcal{J}^*_k$ that is the projection of $\p U^*_k\cap B_{r_0}$ on the $y_n$-axis containing the origin inside.
The function $\rho^*_k$ satisfies \eqref{rhok111} as in dimension two. 

Let $u_k$  be the potential function for the optimal transport from $(U_k, 1)$ to $(U^*_k, g_k)$, 
where the density $g_k=\frac{|U_k|}{|U^*_k|}$ is a constant. 
Subtracting a constant we have $u_k(0)=0$. 
Since $U^*_k$ is convex, as before we can extend  $u_k$ to $\mathbb{R}^n$ by
$$u_k(x):=\sup\{\ell(x) :\ \ell \text{ is affine, } \ell\leq u_k\ \text{in}\ U_k,\ \nabla\ell\in U_k^*\} \quad \mbox{for }x\in\mathbb{R}^n. $$
Since $u_0\in C^1(\R^n)$ and $u_0(0)=0$, by the uniqueness of potential functions,
$u_k \to u_0$ uniformly in $B_{r_0}(0)$ for a different $r_0>0$ small.  
In addition we have $\|u_k-u_0\|_{C^1(B_{r_0/2})} \to0$  as  $k\to\infty.$ 

\begin{lemma}\label{high smooth new}
\begin{itemize}
\item[$(i)$] For each $k\geq1$, we have 
	\begin{equation}\label{11pp}
		\nu_k(x) \cdot \nu_k^*(Du_k(x)) > 0 \quad \forall\, x\in \Gamma_k,
	\end{equation}
where $\nu_k$ and $\nu^*_k$ are the unit inner normals of the domains $\{x\in\R^n : x_n>\rho_k(x_1)\}\cap B_{r_0}$ and $\{y\in\R^n : y_1>\rho^*_k(y_n)\}\cap B_{r_0}$, respectively. 

\item[$(ii)$] For each $k\geq1$, $u_k$ is smooth, locally uniformly convex, and $\det\,D^2u_k$ is a positive constant in $B_{r_0}(0)\cap \{x_n>\rho_k(x_1)\}$ (up to the boundary $\Gamma_k$).
\end{itemize}
\end{lemma} 

\begin{proof}
Similarly as in Lemma \ref{2dapp}, $(ii)$ follows from $(i)$. 
That is, if the obliqueness \eqref{11pp} holds, by \S\ref{s555} we have the smoothness of $u_k$ in $(ii)$.
The proof of $(i)$ will be given in the following two lemmas.  
\end{proof}

 \begin{lemma}\label{sp111}
For any fixed $k\geq1$, assume that $u_k(0)=0$ and $D u_k(0)=0$. 
Then for any $x=(t, x'', \rho_k(t))\in \Gamma_k$ with $t \leq |x''|^{2/3}$,  we have
\beq\label{n470}
 u_k(x)\approx |x''|^2.
\eeq 
\end{lemma}

\begin{proof}
Since the boundaries $\Gamma_k, \Gamma^*_k$ are flat in $x''=(x_2,\cdots,x_{n-1})$, and $\rho_k, \rho^*_k$ are smooth and uniformly convex, we can choose $\eps=0$ in Lemmas \ref{lem4.6} and \ref{co61} (similarly as in Corollary \ref{newcoro}).
From \eqref{mmeq}, we have
	\beq\label{uuin} 
		C^{-1}Q \cap U_k \subset S_h[u_k] \subset CQ \quad\text{ with }\ Q:=[-\tilde q_1, \tilde q_1]\times B_{h^{1/2}}(0) \times [-d_n,d_n], 
	\eeq
where $\tilde q, d_n$ are defined in \eqref{defq1}, \eqref{defq2} respectively. 
Similarly to \eqref{dqq}, let $q, \xi\in\p S_h[u_k]$ be the points on $\p S_h[u_k]$ such that   
\begin{align*}
	q_1 &= \langle q, e_1\rangle = \sup\{\langle x, e_1 \rangle : x\in S_h[u_k]\}, \\
	\xi_n &= \langle \xi, e_n\rangle = \sup\{\langle x, e_n \rangle : x\in S_h[u_k]\}.
\end{align*}  
Since $Du_k(U_k)\subset U^*_k\subset\{y_1\geq0\}$, $u_k$ is increasing in $e_1$ direction. Hence $\xi$ can be chosen on $\Gamma_k$. 
Then by \eqref{uuin}, $C^{-1}\tilde q_1\le q_1\le C\tilde q_1$ (see also Remark \ref{cwidth}) and $C^{-1}d_n<\xi_n\le Cd_n$.
Since $\rho_k\in C^2$, we have $\xi_n\leq C_1\xi_1^2$.
By the uniformly convexity of $\rho_k$, we have $q_n\geq C_2q_1^2$.
By Corollary \ref{rl42}, we then obtain
	$$ \tilde C_1\tilde q_1^2  \geq \xi_n \geq q_n \geq \tilde C_2 \tilde q_1^2,\quad \text{ thus } d_n \approx \tilde q_1^2. $$
By the fact that $|S_h[u_k]|\approx h^{n/2}$, we then have $\tilde q_1\approx h^{1/3}$.
Hence, when $x=(t, x'', \rho_k(t))\in \Gamma_k$ with $t \leq |x''|^{2/3}$, we obtain \eqref{n470}. 
\end{proof}

\begin{remark}\label{hdCaff}
In dimension two, we can use Caffarelli's regularity to conclude that $u_k$ is $C^{2,\alpha}$ smooth up to $\Gamma_k$ (Remark \ref{2dCaff}).
In high dimensions, for the proof of Lemma \ref{lem obli} in \S\ref{sss53}, 
we have to choose the domains $U_k, U^*_k$ which are flat in $e_2,\cdots,e_{n-1}$ directions.
Hence we cannot use Caffarelli's boundary $C^{2,\alpha}$ regularity \cite{C96} directly.
But with the help of Lemma \ref{sp111}, 
one can modify Caffarelli's argument to prove that $u_k$ is smooth up to $\Gamma_k$.
In fact, Lemma \ref{sp111} implies that the solution $u_k$ (for any fixed $k$) behaves nicely in $x''$,
and so the directions $x''$ wouldn't cause us new troubles. 
Here we will not use the argument in \cite{C96} but provide an independent proof of \eqref{11pp}, based on Lemma \ref{1p}. 
\end{remark}

\begin{lemma}\label{ndob new}
For any fixed $k\geq 1$, \eqref{11pp} holds. 
\end{lemma}

\begin{proof} 
Suppose to the contrary that \eqref{11pp} fails at a point $\hat x\in \Gamma_k$, that is
$$ \nu_k(\hat x)\cdot \nu^*_k(D u_k(\hat x))=0. $$
By a change of coordinates and subtracting a linear function,
we can assume $\hat x=0$, $u_k(0)=0$ and $D u_k(0)=0$ such that the hypotheses of Lemma \ref{sp111} are satisfied. 

Consider the restriction of $\p U_k$ in $\text{span}\{e_1, e_n\}$. 
For a boundary point $p=(t,0,\cdots,0, \rho_k(t))\in \Gamma_k$, let $h=u_k(p)$.
Denote by $\tilde\eta$ the unit inner normal of $S_{h}[u_k]$ at $p,$ and $\eta$ the projection of 
$\tilde\eta$ on $\text{span}\{e_1, e_n\}$. 
Denote by $\nu$ the unit inner normal of $\partial U_k$ at $p,$ 
and $\alpha$ the angle between $\eta$ and $\nu$ (see Figure 3). 
Note that by $d)$ in our domain constructions, $\p U_k, \p U^*_k$ are flat along $e_2,\cdots,e_{n-1}$ directions near the origin. Hence the normal vectors $\nu(p), \nu^*(p^*)$ and the tangential vectors $\zeta(p), \zeta^*(p^*)$ are all in the 2-dim plane $\text{span}\{e_1, e_n\}$, where $p^*=Du_k(p)$.  
By the strict convexity of $u_k$ and the proof of Lemma \ref{1p}, 
there exists a small $t_0>0$ such that $\alpha \ge 0$ 
at $p_0=(t_0,0, \rho_k(t_0))$, which implies the obliqueness holds at $p_0$. 
Hence by the $C^{1,\delta}$ regularity,  there is a small constant $\epsilon_0>0$ such that
	\begin{equation*}
		\nu_k(p)\cdot \nu^*_k(Du_k(p))>0,\ \ \ \forall \, p = (t, p'', \rho_k(t))\ \mbox{ with } t\in (t_0-\epsilon_0, t_0] \mbox{ and } |p''| \leq \epsilon_0.
	\end{equation*}

For any $t\in(0,t_0)$, denote
$$\mathcal{C}_t = \left\{ (x_1, x'',0) :  t<x_1<t_0,  \ |x''| < \epsilon_0(x_1-t) \right\},$$ 	
which is an $(n-1)$-dimensional round open cone in the hyperplane $\{x_n=0\}$ with vertex at $(t,0,0)$ and base on the disk 
$\{(t_0, x'', 0) : |x''| \leq \epsilon_0(t_0-t)\}$.

Let 
$$\tilde t=\inf\left\{ t:\ \text{the obliqueness holds } \forall\, p\in \Gamma_k, \text{ provided } (p-p_ne_n)\in\mathcal{C}_t \right\} , $$ 
where $p-p_ne_n$ is the projection of $p$ on the plane $\{x_n=0\}$.
Obviously $\tilde t\geq 0$, and there is a point $(\tilde{x}_1, \tilde{x}'',0)\in \partial \mathcal{C}_{\tilde t}$,
with $\tilde x_1<t_0$,
such that the obliqueness fails at $\tilde x = (\tilde{x}_1, \tilde{x}'',  \rho_k(\tilde{x}_1))$ 
but it holds in $\{(x_1, x'', \rho_k(x_1)):\ (x_1, x'', 0)\in \mathcal C_{\tilde t}\}$.

Therefore by a change of coordinates, we can assume that the obliqueness fails at the origin but it holds 
for all $x\in\Gamma_k$ whose projection $(x-x_ne_n)\in {\mathcal C}_0$, where
${\mathcal C}_0 ={\mathcal C}_{t\,|\,  t=0}$ was the cone defined above.
By subtracting a linear function, we again have $u_k(0)=0$, $Du_k(0)=0$.

Now, we introduce the auxiliary function  
\beq\label{new-w}
w=\partial_1 u_k+K(u_k-\frac{n}{2}x_1\partial_1 u_k) , 
\eeq
where $K$ is a large constant to be determined. 
Let $\underline{w}$ be the function given by
	\begin{equation}\label{new infi} 
		\underline{w}(t) = \inf\{ w(t, x_2, \cdots, x_n) :  (t, x_2, \cdots, x_n) \in  U_k\cap B_{r_0} \}
	\end{equation}
for $t>0$ small. 

We \textit{claim} that the infimum in \eqref{new infi} cannot be attained on the boundary $\p(U_k\cap B_{r_0})$ for $t>0$ small. 
Indeed, as in the proof of Corollary \ref{auco}, there exists a small constant $\tau_0>0$ 
such that for $t\in(0,\tau_0)$, the infimum in \eqref{new infi} cannot be attained on $U_k\cap \p B_{r_0}$.
Hence, it suffices to prove the claim over the part $\Gamma_k=\p U_k\cap B_{r_0}$. 
For any given $0<t<\min\{\tau_0, t_0\}$, denote
	$$\Gamma_k\cap\{x_1=t\} = \partial_{in}(t) \cup \partial_{out}(t), $$
where $\partial_{in}(t)$ denotes the boundary points $x\in\Gamma_k$ whose projection $(x-x_ne_n)\in\mathcal{C}_0$, 
while $\partial_{out}(t)=\Gamma_k\cap\{x_1=t\} - \partial_{in}(t)$.

By our choice of the cone $\mathcal C_0$, 
the obliqueness holds at all points $x\in\partial_{in}(t)$.
Hence $u_k$ is smooth up to the boundary $\partial_{in}(t)$.
Similarly to Lemma \ref{u12}, we then infer that $\partial_{1n} u_k<0$ and $\partial_n w<0$ on $\partial_{in}(t)$.
Hence the infimum in \eqref{new infi}  cannot be attained on $\partial_{in}(t)$. 

Next we show that the infimum cannot be attained on $\partial_{out}(t)$ either. 
On the one hand, since $U_k, U^*_k$ satisfy the condition \eqref{hdom} of \S\ref{sn51}, 
by Corollary \ref{keyco} and \eqref{n470}, when $t>0$ sufficiently small, we have 
$$\underline{u_k}(t) = \inf\{u_k(t, x_2, \cdots, x_n) :  (t, x_2, \cdots, x_n) \in  U_k \cap B_{r_0} \} \le Ct^3.$$ 
Note that due to the flatness of $\p U_k, \p U^*_k$ in $e_2,\cdots,e_{n-1}$ directions, 
we can choose $\eps=0$ in \eqref{keyh}. 
Hence  similarly to \eqref{2dbb} we obtain 
\beq\label{n480}
\underline{w}(t) \le Ct^2+CKt^3.
\eeq
On the other hand, for any point $x=(t,x'',\rho_k(t))\in\partial_{out}(t)$, we have  $|x''|>\eps_0 t$.
Hence by \eqref{n470}, we have $u_k(x)\ge c_0|x''|^2 > c_0\eps_0^2 t^2$. Since $\partial_1u_k\ge 0$, we then obtain that for $t<2(nK)^{-1}$ small,
\beq\label{n490}
w(x) >  K u_k(x) > Kc_0\eps_0^2 t^2\quad \ \forall \, x\in\partial_{out}(t).
\eeq
Therefore, by choosing $K$ sufficiently large, from \eqref{n480} and \eqref{n490} 
one can see that $w(x) > \underline{w}(t)$ for all $x\in\partial_{out}(t)$,
namely the infimum in \eqref{new infi} cannot be attained on $\partial_{out}(t)$. 

Once the claim is proved, 
we can show that $\underline{w}$ {is} concave and reach a contradiction by a similar argument as in \S\ref{ns44}. 
The proof of Lemma \ref{ndob new} is finished. 
\end{proof}

With the preparations in \S\ref{sn51} and \S\ref{hdsm}, we are now in position to prove Lemma \ref{lem obli}. 
 
\subsection{Proof of Lemma \ref{lem obli}}\label{sss53}

By our construction, $0\in\Gamma_k$, $0\in\Gamma^*_k$, and \eqref{rhok1}, \eqref{rhok111} hold for $\rho_k, \rho^*_k$, respectively. 
Similarly to \eqref{1side} we have
\beq \label{n1side}
\partial_{x_n}u_k(t,x'',\rho_k(t))<0 \ \ \ \text{for}\ \ (t,x'',\rho_k(t))\in \Gamma_k\cap\{x_1>0\} \ \ \text{near the origin}.
\eeq

Now, we can prove Lemma \ref{lem obli} in a similar way as in \S\ref{s4}, which is outlined as follows:
\begin{itemize}
\item [$(i)$] 
By the computation as in Lemma \ref{u12}, we have 
\begin{equation}\label{key3h}
		\partial_{_{x_1x_n}}u_k(x)<0\ \ \ \ 
\forall\ x\in \partial U_k \cap\{x\in B_{r_0}: x_1>0\}. 	
\end{equation}

\item [$(ii)$]	Define the auxiliary function
	\begin{equation*} 
	 w_k(x) :=\partial_{_{x_1}} u_k + u_k - \frac{n}{2}x_1 \partial_{_{x_1}}u_k  
	\end{equation*}
that satisfies
	\begin{equation*} 
		M^{ij}D_{ij}w_k=0
	\end{equation*}
in $B_{r_0}\cap {U_k}$, where $\{M^{ij}\}$ is the cofactor matrix of $D^2u_k$. 

\item [$(iii)$]
By \eqref{key3h}, similarly to Corollary \ref{auco},
we see that there exists a constant $\epsilon_0$ independent of $k,$ such that
for any given $t\in(0,\epsilon_0)$,
the function $w_k(t, \cdot)$ has an interior local minimum.
Hence we can define 
	\begin{equation}\label{infi}
		\underline{w_k}(t) = \inf\{ w_k(t, x_2, \cdots, x_n) :  (t, x_2, \cdots, x_n) \in U_k \}  
	\end{equation}
for $t\in(0,\epsilon_0)$.
Note that by \eqref{key3h}, the infimum cannot be attained on $\partial U_k\cap B_{\epsilon_0}$.

\item [$(iv)$] Similarly to Lemma \ref{leco}, we can prove that
$\underline{w_k}$ is concave in $(0, \epsilon_0).$

\item [$(v)$] By letting $k\to\infty$, we have now obtained the function $\underline {w_0}$ that satisfies 
\begin{itemize} 
\item [$a)$] $\underline {w_0}\ge 0$, and $\underline {w_0}(t)\to0$ as $t\to0$.
\item [$b)$]   $\underline {w_0}$ is concave.
\item [$c)$]   $\underline{w_0} (t) \leq Ct^{2-\eps}\ \ \text{ for $t>0$ small}$ \ \ (by \eqref{keyh} that also holds for $u_0$ with the same constants).
\end{itemize}
Therefore $\underline {w_0}\equiv 0$, and we reach a contradiction analogous to that of dimension two. 
This completes the proof of Lemma \ref{lem obli}.
\end{itemize}
\qed

\section{Proof of Theorems \ref{main} and \ref{main01} }\label{s555}

In this section we prove Theorems \ref{main} and \ref{main01}, 
namely the  global $C^{2,\alpha}$ and $W^{2,p}$ estimates for the problem \eqref{MA1}, \eqref{bdry}.
In \cite {C96}, Caffarelli established the global $C^{2,\alpha'}$ estimate for some $\alpha'\in (0, \alpha)$.
The exponent  $\alpha'$ can be improved to  $ \alpha$, 
using the global $C^{2,\alpha}$ estimate for the Dirichlet problem in \cite {TW08a, Sa}. 
Here we give a direct proof.
We also obtain the continuity of $D^2 u$ for Dini continuous and positive $f$. 
 
Assume that $0\in\partial\Omega$, $u(0)=0$ and $Du(0)=0\in\partial\Omega^*$. 
By the uniform obliqueness \eqref{sobli} and a linear transform of the coordinates, 
we may assume that locally  
	\begin{align*}
		\partial\Omega &= \{ x_n= \rho(x') \}, \\ 
		\partial\Omega^* &= \{ y_n= \rho^*(y')\},
	\end{align*}
where $\rho, \rho^*\in C^{1,1}$ satisfying $\rho, \rho^* \geq 0$ and $\rho(0)=\rho^*(0)=0$. 
Note that this expression implies that $u_{x_n}>0$ in $\Om$.

Extending $u$ to $\R^n$ as at the beginning of \S\ref{s2}. Denote
	\begin{equation}\label{newD}
		D^+_{h, a} = \{x\in \R^n:\ u (x)< h\}\cap\{x_n > a\},
	\end{equation}
where $a\geq0$ is a small constant. 
Let $a_h$ be the smallest number such that $D_{h, a_h}^+\subset\Omega$, but $D_{h, a_h-\varepsilon}^+\nsubseteq \Omega$ for any $\varepsilon>0$. One can see that $a_h\to0$ as $h\to0$.
For simplicity, we denote $D^+_{h, a_h}$ by $D^+_h$.

Let $D^-_h$ be the reflection of $D^+_h$ with respect to the plane $\{x_n=a_h\}$, 
and $D_h:=D^+_h\cup D^-_h$.
Since $D_nu \ge 0$, the domain $D_h$ is convex. Moreover, $D_h$ shrinks to the origin as $h\to0$. 

\begin{lemma}\label{newgeo}
The shape of $D_h$ is close to a ball of radius $h^{1/2}$, in the sense that 
	\begin{equation}\label{newinc}
		B_{C^{-1}h^{\frac12+\epsilon}}(x_h) \subset D_h \subset B_{Ch^{\frac12-\epsilon}}(x_h)
	\end{equation}
for any given small $\epsilon>0$, where the centre $x_h=a_he_n$.
\end{lemma}

\begin{proof}
First we show the centred sub-level set $S^c_h[u]$ is close to a ball of radius $h^{1/2}$, namely
	\begin{equation}\label{cball}
		B_{C^{-1}h^{\frac12+\epsilon}}(0) \subset S^c_h[u] \subset B_{Ch^{\frac12-\epsilon}}(0)
	\end{equation}
for any small $\epsilon>0$. 
Indeed, from Lemma \ref{treg},
	\begin{equation}\label{newinc1}
		B_{C^{-1}h^{\frac12+\epsilon}}(0) \cap \{x_n=0\} \subset S^c_h[u] \cap \{x_n=0\}
	\end{equation}
for any small $\epsilon>0$. 
Similarly, this also holds for centred sub-level sets $S^c_h[v]$, for the dual potential $v$.

Let $e'$ be a unit tangential vector with $e'\perp e_n$, and let $t>0$ such that $te' \in \partial S^c_h[v]$. 
Applying \eqref{newinc1} to $v$, we have
$t \ge C^{-1} h^{\frac12+\epsilon}$. 
For any $x\in S^c_h[u]$, from \eqref{inn} 
	\begin{equation}\label{cylin}
		|x \cdot e'| \le C \frac{h}{t} \le Ch^{\frac12-\epsilon},
	\end{equation}	
which implies that $S^c_h[u]$ is contained in a vertical cylinder centred at the origin with radius $r'\le Ch^{\frac12-\epsilon}$, 
for any given small $\epsilon>0$. Hence we have proved that 
\begin{equation*}
		B_{C^{-1}h^{\frac12+\epsilon}}(0)  \cap \{x_n=0\} \subset S^c_h[u] \cap\{x_n=0\}\subset B_{Ch^{\frac12-\epsilon}}(0).
	\end{equation*}

Let $r_ne_n\in \partial S^c_h[u]$ and let $S'_h[u]$ be the projection of $S^c_h[u]$ on $\{x_n=0\}$. 
By the convexity of $u$ and noticing that $u_n\ge 0$,
	\begin{equation*}
		r_n |S'_h[u]| = C \Vol(S^c_h[u]) = Ch^{\frac{n}{2}},
	\end{equation*}	
where for the last equality we use  \eqref{gV1}. By \eqref{cylin}, $|S'_h[u]| \le Ch^{(n-1)(\frac12-\epsilon)}$, and thus we obtain
	\begin{equation}\label{rng}
		r_n \ge C h^{\frac12+\epsilon}
	\end{equation}
for a different small $\epsilon>0$. 
\eqref{rng} is also true for the dual centred sub-level set $S^c_h[v]$. 
By \eqref{newinc1}, we have, similarly to \eqref{cylin},
	\begin{equation}\label{rnl}
		|x\cdot e_n| \le Ch^{\frac12-\epsilon} \ \ \forall\  x\in S^c_h[u].
	\end{equation}
Combining \eqref{newinc1}--\eqref{rnl} we obtain  \eqref{cball}.

Next we show that there exist two constants $b_1, b_2$, independent of $u$ and $h$, such that
	\begin{equation}\label{st2}
		S^{c,+}_{b_1h} \subset D^+_{h,0}  \subset S^{c,+}_{b_2h},
	\end{equation}
where  $S^{c,+}_{h}=S^{c}_{h}[u]\cap\{x_n>0\}$, 
and $D^+_{h,0}$ is given in \eqref{newD}.
{The first inclusion can be proved similarly as that of \eqref{relate1}.
Indeed, for any $x\in \partial S^{c}_{h}\cap\{x_n=0\}$, 
by  \eqref{cball}   and since $\pom\in C^{1,1}$, we have $\dist(x,  \partial S^{c}_{h}\cap\Om) \le Ch^{1-\epsilon}$.
By \eqref{cball}  we also have $|Du| \le Ch^{\frac12-\epsilon}$ in $S^c_h$. Hence
	\begin{equation*}
		u(x) \ge Ch\qquad \forall\ x\in\partial S^{c}_{h}\cap\{x_n=0\}
	\end{equation*}
and \eqref {st2} follows. The second inclusion of \eqref{st2} also follows from \eqref{relate1}. }
 
We are ready to prove \eqref{newinc}.
Combining \eqref{cball} and \eqref{st2}, there exists a constant $C$ independent of $u$ and $h$ such that
	\begin{equation}\label{uphalf}
		B_{C^{-1}h^{\frac12+\epsilon}}(0)\cap\{x_n>0\} \subset D^+_{h,0} \subset B_{Ch^{\frac12-\epsilon}}(0)\cap\{x_n>0\}	
	\end{equation}
for any given small $\epsilon>0$. 
Since $\partial\Om\in C^{1,1}$, by the definition of $a_h$ (after \eqref{newD}), 
one has $a_h<Ch^{1-\epsilon}$ for some $\epsilon>0$ as small as we want.
Recall that $D_h=D^+_{h,a_h}\cup D^-_{h,a_h}$ and $D^-_{h,a_h}$ is an even extension of $D^+_{h,a_h}$ with respect to $\{x_n=a_h\}$.
We obtain   \eqref{newinc} from \eqref{uphalf}. 
\end{proof}

By \eqref{newinc}, we infer that 

\begin{corollary} \label{coro5.1}
For any given small $\eps>0$, $u\in C^{1, 1-\eps}(\bom).$
\end{corollary}

From the above $C^{1,\alpha}$ regularity for all $\alpha<1$, we can prove the global $W^{2,p}$ regularity (Theorem \ref{main01}).
As mentioned in the introduction, the global $W^{2,p}$ estimate for the problem \eqref{MA1}--\eqref{bdry}
was obtained  in \cite{CF}, using the estimates of Caffarelli in \cite{C96}.
Hence the domains are the uniform convexity with $C^2$ boundaries in \cite {CF}.
By our estimates above, we can remove the uniform convexity condition and reduce the smoothness assumption on domains.  

\begin{proof}[Proof of Theorem \ref{main01}]
The proof is based on the estimate \eqref{newinc} and uses the argument of Savin \cite{Sa1}, see also \cite{CF}.
For completeness, let us outline the main steps.
Given $x\in\Om$, let $\bar h_x$ be the maximal value of $h$ such that $S_h[u](x)\subset\Om$, i.e.
	\begin{equation*}
		\bar h_x := \max\{h\geq0 : S_h(x)\subset\Om\}.
	\end{equation*}
Let $T$ be a unimodular linear transform such that $T(S_{\bar h_x}[u](x)) \sim B_{\bar h^{1/2}_x}$.
By \eqref{newinc}, one has $\|T\|, \|T^{-1}\| \lesssim \bar h_x^{-\eps}$, for any small $\eps>0$. 
Hence, when $\bar h_x$ is small
	\begin{equation}\label{sav1}
		S_{\bar h_x}[u](x) \subset D_{C\bar h^{1/2}_x} := \{z\in\overline\Om : \dist(z,\pom) \leq C\bar h^{1/2-\eps}_x\} 
	\end{equation}
for any small $\eps>0$.
	
By subtracting a linear function we may assume that $x=0$, $u(0)=0$ and $Du(0)=0$. 	
Let
	\begin{equation*}
		\tilde u(x) = \frac{1}{\bar h}u(\bar h^{1/2}T^{-1}x),
	\end{equation*}
where $x\in \tilde S_1(0) = \bar h^{-1/2}T(S_{\bar h}(0))$.
The interior $W^{2,p}$ estimate for $\tilde u$ in $\tilde S_1(0)$ gives $\int_{\tilde S_{1/2}(0)}\|D^2\tilde u\|^p dx \leq C$;
hence by rescaling 
	\begin{equation}\label{sav2}
		\int_{S_{\bar h/2}(0)}\|D^2u\|^p\, dx = \int_{\tilde S_{1/2}(0)}\|T'D^2\tilde u T\|^p \bar h^{n/2}\,dx \leq C\bar h^{\frac{n}{2}-2\eps p}.
	\end{equation}	

From Vitali covering lemma, there exists a sequence of disjoint sub-level sets $\{S_{\delta \bar h_i}(x_i)\}$, 
$\bar h_i=\bar h_{x_i}$ such that $\Om\subset\bigcup_{i=1}^\infty S_{\bar h_i/2}(x_i)$, 
where $\delta>0$ is a small constant, (see \cite[Lemma 2.3]{Sa1}).
Then
	\begin{equation*}
		\int_\Om \|D^2u\|^p\,dx \leq \sum_i\int_{S_{\bar h_i/2}(x_i)} \|D^2u\|^p\,dx.
	\end{equation*}

Note that it suffices to consider those $\bar h_i \leq c_1$ for a small constant $c_1>0$.
We can adopt the argument of Savin \cite{Sa1}: Consider the family $\mathcal{F}_d$ of those $S_{\bar h_i/2}(x_i)$ satisfying 
	\begin{equation*}
		d/2 < \bar h_i \leq d
	\end{equation*} 
for a constant $d\leq c_1$. By \eqref{sav2} and \eqref{gV2}
	\begin{equation*}
		\int_{S_{\bar h_i/2}(x_i)} \|D^2u\|^p \, dx \leq C d^{-2\eps p} |S_{\delta\bar h_i}(x_i)|.
	\end{equation*}
By \eqref{sav1} and Vitali covering lemma, $S_{\delta\bar h_i}(x_i)\subset D_{Cd^{1/2-\eps}}$ and are disjoint. 
Hence, we have
	\begin{equation*}
		\sum_{i\in\mathcal{F}_d}\int_{S_{\bar h_i/2}(x_i)} \|D^2u\|^p \, dx \leq C d^{\frac12-\eps-2\eps p} \leq C d^{\frac12-3\eps p}.
	\end{equation*}
Let $d=c_12^{-k}$, $k=0,1,2,\cdots$, and by adding the sequence of inequalities, we obtain 	
	\begin{equation*}
		\int_\Om \|D^2u\|^p\,dx \leq C + C_1\sum_{k=0}^\infty 2^{-k(\frac12-3\eps p)}.
	\end{equation*}
For any $p\geq1$, as $\eps$ is arbitrarily small so that $3\eps p < \frac14$, therefore the series is convergent. 
\end{proof}

Now we continue with the proof of global $C^{2,\alpha}$ estimate. 
Let $w$ be the solution of 
	  \begin{eqnarray*}
	     \det\, D^2 w=1 &&\mbox{in } D_h , \\
	     w=h &&\mbox{on }\partial  D_h.
	  \end{eqnarray*}  	  
Denote by $\tilde u$  the even extension of $u$ with respect to $\{x_n=0\}$, namely
	\begin{equation*}
		\tilde u(x', x_n) = \Big\{\begin{array}{ll}
			u(x',x_n) & \quad\mbox{ if } x_n\ge0 ,  \\
			u(x',-x_n) & \quad\mbox{ if } x_n<0.
		\end{array}
	\end{equation*}
For simplicity, we still denote $\tilde u$ by $u$.
The following lemma gives an estimate on the difference between the \textit{``original"} solution $u$ and the \textit{``good"} solution $w$.
	  
\begin{lemma}\label{lemL}
Assume $|f-1| \le h^\delta$ in $D_h\cap\Om$ for some $\delta\in(0,1/2)$. We have
	\begin{equation}\label{wL}
		|u-w| \le Ch^{1+\delta}\qquad\mbox{in } D_h\cap\Om,
	\end{equation}
where the constant $C$ is independent of $h, \delta$. 
\end{lemma}	  
	  
\begin{proof}
Divide $\partial D_h^+ = \mathcal{C}_1\cup\mathcal{C}_2$ into two parts, where $\mathcal{C}_1\subset\{x_n> a_h\}$ and $\mathcal{C}_2\subset\{x_n= a_h\}$. 
On $\mathcal{C}_1$ we have $u=w$. 
On $\mathcal{C}_2$, by symmetry we have $D_nw=0$. 
As $a_h<Ch^{1-\epsilon}$, by Corollary \ref{coro5.1} we have $0\leq D_nu \leq C_1h^{1-\epsilon}$ on $\mathcal{C}_2$, 
for any given small $\epsilon>0$.

Let 
$$\hat w = (1-h^\delta)^{1/n}w-(1-h^\delta)^{1/n}h+h. $$ 
Then
	\begin{align*}
		\det\,D^2\hat w \leq \det\,D^2 u & \quad\mbox{ in }  D_h^+,\\
		\hat w = u = h &\quad\mbox{ on }\mathcal{C}_1, \\
		D_n\hat w =0 < D_n u &\quad\mbox{ on }\mathcal{C}_2.
	\end{align*}
By comparison principle we have $\hat w \geq u$ in $D_h^+$. 

On the other hand, let 
$$\check w = (1+h^\delta)^{1/n}w-(1+h^\delta)^{1/n}h+h+C_1(x_n-Ch^{1/2-\epsilon})h^{1-\epsilon}. $$
Then 
	\begin{align*}
		\det\,D^2\check w \geq \det\,D^2 u & \quad\mbox{ in }  D_h^+,\\
		\check w \leq u =h &\quad\mbox{ on }\mathcal{C}_1, \\
		D_n\check w = C_1h^{1-\epsilon} > D_n u &\quad\mbox{ on }\mathcal{C}_2.
	\end{align*}
Hence by comparison principle,  $\check w \leq u$ in $ D_h^+$.

Since $h>0$ is small, $\delta<1/2$, and $\epsilon>0$ is  small, we obtain
	\begin{equation}\label{coreL}
		|u-w| \leq Ch^{1+\delta} \qquad\mbox{in } D_h^+. 
	\end{equation}

Next, we estimate $|u-w|$ in $ D_h^-\cap\hskip1.5pt \Omega$. 
For $x=(x',x_n)\in D_h^-\cap\hskip1.5pt \Omega$, let $z=(x', 2a_h-x_n) \in D_h^+$.
Then  $|x-z|\leq Ch^{1-\epsilon}$. 
From \eqref{coreL}, $|u(z)-w(z)| \leq Ch^{1+\delta}$.
Since $w$ is symmetric with respect to $\{x_n=a_h\}$, we have  $w(x)=w(z)$.
Since $u\in C^{1,1-\epsilon}(\bom)$, we obtain
	\begin{equation*}
		|u(x)-u(z)| \leq \|Du\|_{L^\infty(D_h)} |x-z| \leq Ch^{3/2-\epsilon}.
	\end{equation*}
Therefore, as $\delta<1/2$ is a given constant 
	\begin{equation*}
		|u(x)-w(x)| \leq |u(x)-u(z)| + |u(z)-w(z)| \leq Ch^{1+\delta}.
	\end{equation*}
Combining with \eqref{coreL} we thus obtain the desired $L^\infty$ estimate
	\begin{equation}\label{coreLL}
		|u-w| \leq Ch^{1+\delta}\qquad\mbox{in }  D_h\cap\Omega. 
	\end{equation}

\vskip-25pt \end{proof}

We are now in position to prove the global $C^{2,\alpha}$ estimate.
We will adopt the argument in  \cite{JW}.
Note that when $f$ is H\"older continuous with exponent $\alpha\in (0,1)$ 
and $f(0)=1$, from Lemma \ref{newgeo} the oscillation 
	\begin{equation}\label{5002}
		\omega_f(h) := \sup_{D_h\cap\Om}|f-1|\leq Ch^\delta   
	\end{equation}
for some
	\begin{equation}\label{5014}
	 \delta\ge \alpha/2-\eps ,
	 \end{equation}
where $\eps>0$ is a small constant arising in \eqref{newinc}.
We point out that if $\eps=0$ in \eqref{newinc}, then $\delta=\alpha/2$.
We first quote two lemmas from \cite{JW}.

\begin{lemma}\label{pl2}
Let $u\in C^2$ be a convex solution of $\det\, D^2u=1$ in $\mathcal D$, vanishing on $\partial\mathcal D$.
Suppose $u$ attains its minimum at the origin, and $|D^2u(0)|\leq C_0$ for some constant $C_0>0$.  
Then the domain $\mathcal D$ is of good shape.  
\end{lemma}

\begin{lemma}\label{pl3}
Let $u_i$, $i=1,2$, be two convex solutions of $\det\, D^2u=1$ in $B_1(0)$.
Suppose $\|u_i\|_{C^4}\leq C_0$.
Then if $|u_1-u_2|\leq\delta$ in $B_1(0)$ for some constant $\delta>0$, we have, for $1\leq k\leq 3$,
	\[ |D^k(u_1-u_2)| \leq C\delta \quad\mbox{in } B_{1/2}. \]
\end{lemma}

\begin{proof}[\bf Proof of Theorem \ref{main}]
We sketch the proof here as it is similar to that in \cite {JW}.
Choose a sufficiently small initial height $h_0>0$, and normalise $D_{h_0}$ by a transformation $T$ such that
$T(D_{h_0})\sim B_1(0)$.

After the change, $D_1$ has a good shape. Denote
	\begin{equation}\label{osci1}
		\omega(h) = \omega_f(h),\qquad \omega_k := \omega(4^{-k})
	\end{equation}
for $k=0,1,2,\cdots$, where $\omega_f(h)$ is given in \eqref{5002}.
Define
	\begin{equation}\label{osci2}
		\mathcal D_k := D_{4^{-k}}, \qquad f_k := \inf_{\mathcal D_k\cap \Omega} f > 0.
	\end{equation}

We \emph{claim} that $\mathcal D_k \sim \mathcal D_{k+1}$, namely there is a constant $C$ depending only on $n$ such that 
\beq\label{anmm}
	C^{-1}\mathcal D_k \subset \mathcal D_{k+1} \subset C\mathcal D_k.
\eeq
To see this, note that (before the change $T$) by \eqref{newinc} and $a_{h_0} < Ch_0^{1-\epsilon}$, one has $|D_{h_0}|\approx |S_{h_0}[u]| \approx h_0^{n/2}$ and $|D_{h_0/4}\cap\{x_n>a_{h_0}\}| \approx h_0^{n/2}$. 
Since $|\det\,T|\approx h_0^{-n/2}$, we have $|T(D_{h_0/4}\cap\{x_n>a_{h_0}\})|\approx 1$.
By definition $D_{h_0/4}\cap\{x_n>a_{h_0}\} \subset D_{h_0}$, thus $T(D_{h_0/4}\cap\{x_n>a_{h_0}\})$ is bounded.
Therefore, by symmetry and the fact $a_{h_0} < Ch_0^{1-\epsilon} \ll h^{\frac12+\epsilon}$ that is the width of $D_{h_0/4}$ in $e_n$ direction, we obtain $\mathcal D_1 \sim \mathcal D_0$. Similarly we can obtain \eqref{anmm} for all $k=0,1,2,\cdots$.

Let $u_k$, $k=0,1,2,\cdots$, be the convex solution of
	\begin{align}
		\det\,D^2 u_k = f_k &\qquad\mbox{in }\mathcal D_k,\\
		u_k = 4^{-k} & \qquad\mbox{on }\partial \mathcal D_k. \nonumber
	\end{align}

When $k=0$, since initially $\mathcal D_0$ has a good shape, by interior regularity \cite{GT},
	\begin{equation*}
		\|u_0\|_{C^4(D_{3/4})}\leq C.
	\end{equation*}
From Lemma \ref{lemL}, 
	\begin{equation*}
		\sup_{\mathcal D_0\cap \Om}|u-u_0| \leq C\omega_0,
	\end{equation*}
which implies that $\mathcal D_1$ has a good shape (also shown in \eqref{anmm}), and thus
	\[ \|u_1\|_{C^4(D_{3/16})} \leq C. \]
Hence, from Lemma \ref{pl3} and \eqref{anmm}
	\begin{equation}\label{in1}
		|D^2u_0(x)-D^2u_1(x)| \leq C\omega_0
	\end{equation}
for $x\in C^{-1}\mathcal D_2$, where $1\leq k\leq 3$. By Lemma \ref{pl2}, this estimate then implies that $\mathcal D_2$ has a good shape.

By induction and \eqref{wL}, \eqref{anmm}, we obtain
	\begin{equation}\label{in2}
		|D^2u_k(x)-D^2u_{k+1}(x)| \leq C\omega_k
	\end{equation}
for $x\in C^{-1}\mathcal D_{k+2}$.

Therefore, for any given point $z\in\bom$ near the origin such that $4^{-k-4}\leq u(z)\leq 4^{-k-3}$,
	\begin{equation}\label{in3}
	\begin{split}
		& |D^2u(z)-D^2u(0)| \leq I_1 + I_2 + I_3 := \\
		& |D^2u_k(z)-D^2u_k(0)| + |D^2u_k(0)-D^2u(0)| + |D^2u(z)-D^2u_k(z)|. 
	\end{split}
	\end{equation}
By \eqref{in2},
	\begin{equation}\label{in4}
		I_2 \leq C\sum_{j=k}^\infty \omega_j \leq C\int_0^{|z|} \frac{\omega(r)}{r}.
	\end{equation}
Similarly to \eqref{in4}, as in \cite{JW} one can derive that
	\begin{equation}\label{in44}
		I_3 \leq C\int_0^{|z|} \frac{\omega(r)}{r}.
	\end{equation}	
To estimate $I_1$, denote $h_j=u_j-u_{j-1}$. By Lemma \ref{pl3},
	\begin{equation*}
		\left| D^2h_j(z) - D^2h_j(0) \right| \le C2^j\omega_j|z|.
	\end{equation*}
Hence
	\begin{equation}\label{in5}
	\begin{split}
		I_1 &\le |D^2u_0(z)-D^2u_0(0)| + \sum_{j=1}^k |D^2h_j(z) - D^2h_j(0)| \\		
			&\leq C|z|\Big( 1+ \int_{|z|}^1\frac{\omega(r)}{r^2} \Big).
	\end{split}
	\end{equation}
Since $f$ is H\"older continuous with exponent $\alpha$, 
inserting \eqref{in4}--\eqref{in5} into \eqref{in3} we obtain the H\"older continuity at the origin, namely for any point $z$ near the origin,
	\begin{equation}\label{newS}
		|D^2u(z)-D^2u(0)| \le C|z|^\alpha.
	\end{equation}

In \eqref{newS}, we obtained the H\"older continuity of $D^2 u$ at the boundary. 
If two points $x, y\in\Om$ are both interior points, let $\hat x, \hat y\in\partial\Om$ be the closest points to $x, y$, respectively.  
In the case $|x-y| \ge \delta_0\big(\dist(x,\partial\Om) + \dist(y,\partial\Om)\big)$ for some constant $\delta_0>0$, 
 by \eqref{newS} we have
	\begin{equation*} 
	\begin{split}
		|D^2u(x)-D^2u(y)| \leq\, & |D^2u(x)-D^2u(\hat x)|  \\
			& + |D^2u(\hat x)-D^2u(\hat y)| + |D^2u(\hat y)-D^2u(y)|  \le C|x-y|^\alpha. 
	\end{split}
	\end{equation*}
Otherwise, the estimate for $|D^2u(x)-D^2u(y)|$ has been established in  \cite{C1, JW} for the interior $C^{2,\alpha}$ regularity.

We have proved the global $C^{2,\alpha-\eps}$ regularity for problem \eqref{MA1}, \eqref{bdry}.
To remove the small constant $\eps$, observe that once the second derivatives are uniformly bounded, 
the inclusions \eqref{newinc} holds for $\eps=0$. Therefore \eqref{5014} can be improved to $\delta=\alpha/2$.
Repeating the above argument, we then obtain the global  $C^{2,\alpha}$ regularity for problem \eqref{MA1}, \eqref{bdry}.
\end{proof}

\begin{remark}\label{r51}
The above argument also implies that if $f$ is Dini continuous, that is if 
	\begin{equation*}
		\int_0^1\frac{\omega_f(t)}{t}\,dt <\infty,
	\end{equation*}
where $\omega(t)=\sup\{|f(x)-f(y)| : |x-y|<t\}$, then the integrals in \eqref{in4} and \eqref{in5} are convergent. 
Hence $D^2 u$ is positive definite and continuous up to the boundary. Therefore we  have proved the following result.
\end{remark}

\begin{theorem}\label{main1}
Assume that $\Omega$ and $\Omega^*$ are bounded convex domains in $\mathbb{R}^n$ with $C^{1,1}$ boundaries, 
and  assume that $f$ is positive and Dini continuous.
Then the second derivatives of the solution $u$ to the problem \eqref{MA1} and \eqref{bdry} are continuous in $\bom$.
\end{theorem}

\begin{remark}\label{r52}
Checking the proof of the uniform density (Lemma \ref{luni}),
the tangential $C^{1,\alpha}$ regularity (Lemma \ref{treg}),
the uniform obliqueness (Lemma \ref{lem obli}),
we find  that the $C^{1,1}$ regularity of the boundaries $\pom$ and $\pom^*$
can be weakened to $C^{1,1-\theta}$ for some $\theta>0$ depending on the constant $\delta$,
provided $u$ is globally $C^{1,\delta}$ smooth \cite{C92}.
Therefore  our main result, Theorem \ref{main}, 
holds for $C^{1,1-\theta}$ convex domains $\Omega, \Omega^*$.
In particular, we prove that it suffices to assume $\pom, \partial\Omega^*\in C^{1,\alpha}$ in dimension two \cite{CLW2}. 
When $f\equiv1$, very recently Savin and Yu \cite{SY} obtained the global $W^{2,p}$ estimate for arbitrary bounded convex domains $\Om, \Om^*\subset\mathbb{R}^2$. 
In general dimensions, it may be possible to relax the $C^{1,1}$ regularity of the boundaries to $C^{1,\alpha}.$ Indeed, if one can manage this relaxation for the uniform density estimate and the tangential $C^{1,1-\varepsilon}$ estimate (for all $\varepsilon>0$), then our method for the uniform obliqueness can be applied.
\end{remark}

\begin{remark}\label{r53}
From \cite [\S7.3]{MTW} it is known that for arbitrary positive and smooth functions $f$, 
the convexity of domains is necessary for the global $C^1$ regularity.
However, for a fixed function $f>0$,  by Theorem \ref{main} and a perturbation argument, 
we can prove that \cite{CLW3} the solution is smooth up to the boundary, 
if the domains $\Om$ and $\Om^*$ are smooth perturbations of convex ones, 
even though they are not convex themselves. 
\end{remark}


\end{document}